\documentclass[11pt]{article}

\usepackage{geometry}
\usepackage{jhtitle,amsbsy,amsmath,amsthm,amssymb,times,graphicx}
\usepackage{probmacros}
\usepackage{natbib}
\usepackage{enumerate}
\usepackage{float}
\usepackage{authblk}
\usepackage{tikz-cd}

\newcommand{\vertiii}[1]{{\left\vert\kern-0.25ex\left\vert\kern-0.25ex\left\vert #1 
		\right\vert\kern-0.25ex\right\vert\kern-0.25ex\right\vert}}

\def\baro{\vskip  .2truecm\hfill \hrule height.5pt \vskip  .2truecm}
\def\barba{\vskip -.1truecm\hfill \hrule height.5pt \vskip .4truecm}

\newtheorem{theorem}{Theorem}
\newtheorem{lemma}{Lemma}

\newtheorem{remark}{Remark}
\newtheorem{proposition}{Proposition}

\newcommand{\Z}{{\mathsf{Z}}}

\author[1]{Grant Backlund}
\author[1]{James P. Hobert}
\author[2]{Yeun Ji Jung}
\author[1]{Kshitij Khare}
\affil[1]{Department of Statistics, University of Florida}
\affil[2]{Model Governance Group, JPMorgan Chase \& Co (New York)}
\date{December 2019}
\begin{document}

\title{A Hybrid Scan Gibbs Sampler for Bayesian Models \\ with Latent
  Variables}


\maketitle

\begin{abstract}
Gibbs sampling is a widely popular Markov chain Monte Carlo algorithm
that can be used to analyze intractable posterior distributions
associated with Bayesian hierarchical models.  There are two standard
versions of the Gibbs sampler: The systematic scan (SS) version, where
all variables are updated at each iteration, and the random scan (RS)
version, where a single, randomly selected variable is updated at each
iteration.  The literature comparing the theoretical properties of SS
and RS Gibbs samplers is reviewed, and an alternative \textit{hybrid
  scan} Gibbs sampler is introduced, which is particularly well suited
to Bayesian models with latent variables.  The word ``hybrid''
reflects the fact that the scan used within this algorithm has both
systematic and random elements.  Indeed, at each iteration, one
updates the entire set of latent variables, along with a randomly
chosen block of the remaining variables.  The hybrid scan (HS) Gibbs
sampler has important advantages over the two standard scan Gibbs
samplers.  Firstly, the HS algorithm is often easier to analyze from a
theoretical standpoint.  In particular, it can be much easier to
establish the geometric ergodicity of a HS Gibbs Markov chain than to
do the same for the corresponding SS and RS versions.  Secondly, the
sandwich methodology developed in \citet{hobert2008theoretical}, which
is also reviewed, can be applied to the HS Gibbs algorithm (but not to
the standard scan Gibbs samplers).  It is shown that, under weak
regularity conditions, adding sandwich steps to the HS Gibbs sampler
always results in a theoretically superior algorithm.  Three specific
Bayesian hierarchical models of varying complexity are used to
illustrate the results.  One is a simple location-scale model for data
from the Student's $t$ distribution, which is used as a pedagogical
tool.  The other two are sophisticated, yet practical Bayesian
regression models.
\end{abstract}

\section{Introduction}
\label{sec:intro}

Let $f_U : \mathsf{U} \to [0,\infty)$ be an intractable target
density, and suppose that
$f : \mathsf{U} \times \mathsf{Z} \to [0,\infty)$ is a joint density
whose $u$-marginal is the target, i.e.,
$\int_{\mathsf{Z}} f(u,z) \, dz = f_U(u).$ Think of $U$ as the
parameters in a Bayesian model, and $Z$ as latent data. If
straightforward sampling from the associated conditional densities is
possible, then we can use the data augmentation (DA) algorithm to
explore $f_U.$ Of course, running the algorithm entails alternating
between draws from $f_{Z | U}$ and $f_{U | Z},$ which simulates the
Markov chain whose Markov transition density (Mtd) is
\[
k_{\text{DA}} (u' | u) = \int_{\mathsf{Z}} f_{U | Z} (u'
| z) f_{Z|U}(z | u) \, dz \;.
\]
It's easy to see that $k_{\text{DA}} (u' | u) f_U(u)$ is
symmetric in $(u, u'),$ so the DA Markov chain is reversible with
respect to $f_U$.

To fix ideas, we introduce a simple example.  Let $W_1, \ldots, W_m$
be a random sample from the location-scale Student's $t$ distribution
with known degrees of freedom, $\nu > 0$, and consider a Bayesian
model with prior density given by
$\pi(\mu,\sigma^2) \propto \sigma^{-2} I_{\mathbb{R}_+}(\sigma^2)$,
where $\reals_+ := (0,\infty)$.  This is a standard diffuse prior for
location-scale problems.  We assume throughout that $m \geq 2$, which
is necessary and sufficient for posterior propriety.  The resulting
posterior density is an intractable bivariate density characterized by
\begin{equation}
\label{eq:toyposterior}
  f_U(\mu, \sigma^2) \propto (\sigma^2)^{- \frac{m+2}{2}}
  \prod_{i=1}^m \left( 1 + \frac{(w_i - \mu)^2}{\nu \sigma^2}
  \right)^{-(\nu + 1)/2} I_{\mathbb{R}^+}(\sigma^2) \;.
\end{equation}
So here the role of $U$ is being played by $(\mu,\sigma^2)$, and, in
order to keep the notation consistent, we are suppressing the
dependence on the data, $w = (w_1,\ldots,w_m)$.  There is a simple DA
algorithm for this problem that is based on the standard
representation of a Student's $t$ variate in terms of normal and
$\chi^2$ variates.  Conditional on $(\mu,\sigma^2)$, let
$(W_1,Z_1),\ldots,(W_m,Z_m)$ be independent and identically
distributed (iid) pairs such that, for $i=1,\ldots,m$,
\begin{align*}
& W_i | Z_i , \mu , \sigma^2 \sim \text{N}(\mu, \sigma^2 / z_i) \\ &
  Z_i |\mu, \sigma^2 \sim \text{Gamma}(\nu / 2, \nu / 2) \;.
\end{align*}
Letting $z = (z_1,\ldots,z_m) \in \reals^m_+$, the joint density of
$\{ (W_i,Z_i) \}_{i=1}^m$ is given by
\[
p(w,z | \mu, \sigma^2) = \prod_{i=1}^m \frac{\sqrt{z_i}}{\sqrt{2 \pi
    \sigma^2}} \exp \left\lbrace -\frac{z_i}{2 \sigma^2} (w_i - \mu)^2
\right\rbrace \frac{\left( \frac{\nu}{2}
  \right)^{\frac{\nu}{2}}}{\Gamma \left( \frac{\nu}{2} \right)}
z_i^{\frac{\nu}{2} - 1} \exp \left\lbrace - \frac{\nu z_i}{2}
\right\rbrace I_{\mathbb{R}^+}(z_i)\;.
\]
Now define $f(\mu,\sigma^2,z) \propto p(w,z | \mu, \sigma^2)
\pi(\mu,\sigma^2)$.  Then it's easy to see that
\[
\int_{\reals^m_+} f(\mu,\sigma^2,z) \, dz = f_U(\mu, \sigma^2) \;,
\]
so $z$ is indeed latent data.  It is straightforward to sample from
$f_{Z|U}(z|\mu,\sigma^2)$ since the $Z_i$ are conditionally
independent, each with a gamma distribution.  It's also easy to sample
from $f_{U|Z}(\mu,\sigma^2|z)$ (sequentially) because both
$f_{\sigma^2|Z}(\sigma^2|z)$ and $f_{\mu|\sigma^2,Z}(\mu|\sigma^2,z)$
have simple forms \citep[see, e.g.,][Section
  10.1]{hobert2011handbook}.  We will return to this simple example
several times in order to illustrate certain ideas without having to
wade through the heavy notation associated with a more sophisticated
model.  Now back to the general case.

There are many situations in which useful latent data exist, but the
DA algorithm is not directly applicable.  Specifically, it is often
the case that it \textit{is} possible to draw from $f_{Z|U},$ but it
is not possible to draw from $f_{U|Z}.$ On the other hand, in such
cases, one can sometimes break $u$ into two pieces, $u = (x,y)$, where
$x \in \mathsf{X}, y \in \mathsf{Y}, \mathsf{U} = \mathsf{X} \times
\mathsf{Y},$ in such a way that one is able to draw from $f_{X|Y,Z}$
and from $f_{Y|X,Z}$.  In such cases, we can run a three-block Gibbs
sampler based on $f_{X|Y,Z}$, $f_{Y|X,Z}$ and $f_{Z|X,Y}$.  Of course,
the random scan (RS) version of this Gibbs sampler is reversible,
while the systematic scan (SS) version is not.

Consider again our simple Student's $t$ example, and suppose we change
the prior to
\[
\pi^*(\mu,\sigma^2) \propto \sigma^{-2} \exp \Big \{ -\frac{1}{2} (\mu -
\gamma)^2 \Big \} I_{\mathbb{R}^+}(\sigma^2) \;,
\]
where $\gamma \in \mathbb{R}$ is fixed.  In this case,
$(\mu,\sigma^2)$ is playing the role of $(x,y)$.  Under the new prior,
drawing from $(\mu,\sigma^2)|z$ is no longer straightforward, because
the distribution of $\sigma^2|z$ is nonstandard.  Hence, while the DA
algorithm is still technically implementable (using, say, a rejection
sampler for $\sigma^2|z$), it is much less attractive under the new
prior.  On the other hand, the conditional densities of
$\mu|(\sigma^2,z)$, $\sigma^2|(\mu,z)$, and $z|(\mu,\sigma^2)$ all
have standard forms, so the three-block Gibbs sampler would be easy to
run.

In this paper, we consider an alternative to the SS and RS three-block
Gibbs samplers.  We call it the \textit{hybrid scan Gibbs sampler}.
Fix $r \in (0,1)$ to play the role of a selection
probability. Consider a Markov chain $\{(X_n,Y_n) \}_{n=0}^{\infty}$
with state space $\mathsf{X} \times \mathsf{Y}$ that evolves as
follows. If the current state is $(X_n,Y_n) = (x,y),$ then we simulate
the new state, $(X_{n+1},Y_{n+1}),$ using the following two-step
procedure.

\baro \vspace*{2mm}
\noindent {\rm Iteration $n+1$ of the hybrid scan Gibbs sampler:}
\begin{enumerate}
\item Draw $Z \sim f_{Z|X,Y}(\cdot|x,y),$ call the result $z,$ and,
  independently, draw $W \sim \text{Uniform}(0,1).$

\item 
\begin{enumerate}
\item If $W \le r$, draw $X^* \sim f_{X|Y,Z}(\cdot|y,z),$ and set
  $(X_{n+1},Y_{n+1}) = (X^*,y)$.

\item Otherwise, draw $Y^* \sim f_{Y|X,Z}(\cdot|x,z),$ and set
  $(X_{n+1},Y_{n+1}) = (x,Y^*)$.
\end{enumerate}
\end{enumerate}
\vspace*{-3mm}
\barba
\bigskip

A standard SS Gibbs sampler based on $f_{X|Y,Z}$, $f_{Y|X,Z}$ and
$f_{Z|X,Y}$ updates all three components (in the same prespecified
order) at each iteration.  To run the RS version, we first fix three
selection probabilities $r_1, r_2, r_3 \in (0,1)$ such that
$r_1+r_2+r_3=1$.  Then, at each iteration, we draw from exactly one of
the three full conditionals according to the probabilities $r_1$,
$r_2$ and $r_3$, and leave the remaining two components fixed.  So
hybrid scan (HS) Gibbs can be viewed as a compromise between these
standard Gibbs samplers in the sense that, at each iteration of HS
Gibbs, exactly two of the three full conditionals are sampled.  The
idea of including both systematic and random scan ingredients in a
single Markov chain Monte Carlo (MCMC) algorithm is not new
\citep[see, e.g.,][]{levine2005note}, but we believe that this is the
first concentrated study of this particular algorithm.

The reader will note that we have yet to demonstrate that the HS Gibbs
sampler is actually valid.  In fact, it follows directly from one of
the general results in Appendix~\ref{app:theory} (Proposition 1) that
the Markov chain associated with HS Gibbs is reversible with respect
to $f_{X,Y}(x,y)$ for any $r \in (0,1)$, which implies that the
algorithm is valid.  Proposition 1 is rather technical, and its proof
is based on operator theory.  Fortunately, there is a much simpler way
to establish the desired reversibility.  Indeed, we now show that HS
Gibbs is equivalent to a RS, variable-at-a-time Metropolis-Hastings
(MH) algorithm (in which every proposal is accepted).  It then follows
immediately from basic MCMC theory that the Markov chain associated
with HS Gibbs is reversible with respect to $f_{X,Y}(x,y)$ \citep[see,
  e.g.,][]{geyer2011handbook}.  First, it's clear by inspection that a
single iteration of HS Gibbs can be recast as follows: Suppose the
current state is $(X_n,Y_n) = (x,y)$.  Flip an ``$r$-coin.''  If the
coin comes up heads, then set $(X_{n+1},Y_{n+1}) = (X',y)$, where $X'$
is drawn from the density
\[
c_1(x'|x;y) = \int_{\Z} f_{X|Y,Z}(x'|y,z) f_{Z|X,Y}(z|x,y) dz \;.
\]
If, on the other hand, the coin comes up tails, then set
$(X_{n+1},Y_{n+1}) = (x,Y')$, where $Y'$ is drawn from the density
\[
c_2(y'|y;x) = \int_{\Z} f_{Y|X,Z}(y'|x,z) f_{Z|X,Y}(z|x,y) dz \;.
\]
Now consider a MH algorithm in which, at each iteration, with
probability $r$ we (keep $y$ fixed and) perform a MH update that
leaves $f_{X|Y}(\cdot|y)$ invariant using candidate density
$c_1(x'|x;y)$, and with probability $1-r$ we (keep $x$ fixed and)
perform a MH update that leaves $f_{Y|X}(\cdot|x)$ invariant using
candidate density $c_2(y'|y;x)$.  The Hastings ratio for the update
that leaves $f_{X|Y}(\cdot|y)$ invariant is given by
\[
r(x,x') = \frac{f_{X|Y}(x'|y) c_1(x|x';y)}{f_{X|Y}(x|y) c_1(x'|x;y)}
\;.
\]
The numerator of $r(x,x')$ can be written as
\begin{equation*}
\frac{1}{f_Y(y)} f_{X,Y}(x',y) \int_{\Z} f_{X|Y,Z}(x|y,z)
f_{Z|X,Y}(z|x',y) dz = \frac{1}{f_Y(y)} \int_{\Z} \frac{f(x,y,z)
  f(x',y,z)}{f_{Y,Z}(y,z)} dz \;,
\end{equation*}
which is clearly a symmetric function of $(x,x')$.  Hence, $r(x,x')
\equiv 1$, so the candidate is never rejected.  A similar argument
shows that, for fixed $x$, the MH update for $f_{Y|X}(\cdot|x)$ using
candidate density $c_2(y'|y;x)$ also never rejects.  Therefore, the HS
Gibbs sampler is, in fact, a RS, variable-at-a-time
Metropolis-Hastings (MH) algorithm, and reversibility follows.

As mentioned above, the HS algorithm can be viewed as a compromise
between SS and RS Gibbs.  Now, if it were known that one of the
standard scans always produces a superior Gibbs sampler, then it might
not make sense to consider such a compromise.  However, as we now
explain, this is far from being the case.  There are two standard
criteria for comparing MCMC algorithms.  The first is based on the
convergence rates of the underlying Markov chains, and the second is
based on the asymptotic variances of ergodic averages.
(Appendix~\ref{app:theory} contains some general theory on this topic
for reversible chains.)  It is known that neither of the standard scan
Gibbs samplers dominates the other in terms of convergence rate.
Indeed, there are examples in the literature of SS Gibbs samplers that
converge faster than their RS counterparts, and others where the
opposite is true \citep[see,
e.g.,][]{roberts2015surprising,robe:sahu:1997,he:desa:mitl:re:2016}.
On the other hand, there is some general theory suggesting that the SS
Gibbs sampler is better when the criterion is asymptotic variance, but
these results are rather limited in scope.  For example,
\citet{greenwood1998information} show that the asymptotic variance
under the RS algorithm is no more than twice that under the SS
algorithm, and \citet{andrieu2016random} proves that, when the Gibbs
sampler has exactly two blocks, the SS algorithm is always better.
(See also \citet{maire2014comparison}, \citet{liu1995covariance} and
\citet{latuszynski2013adaptive}.)  So, based on what is currently
known, there is no clear cut winner between the SS and RS versions of
the Gibbs sampler when there are more than two blocks.

The HS Gibbs sampler has important advantages over the two standard
scan Gibbs samplers.  Firstly, it can be much easier to establish the
geometric ergodicity of a HS Gibbs Markov chain than to do the same
for the corresponding systematic and random scan Gibbs chains.  We
provide examples of this in Sections~\ref{sec:glmm} and~\ref{sec:smn}.
(Of course, the important practical benefits of basing one's MCMC
algorithm on a geometrically ergodic Markov chain have been
well-documented by, e.g., \citet{roberts1998markov},
\citet{jones2001honest}, \citet{flegal2008markov} and
\citet{latuszynski2013nonasymptotic}.)  Secondly, as we explain in
Section~\ref{sec:adding}, the sandwich methodology of
\citet{hobert2008theoretical} can be applied to the HS Gibbs algorithm
(but not to the standard scan Gibbs samplers).  This allows for the
addition of up to two extra steps at each iteration that can
potentially speed up the convergence rate without adding much to the
computational complexity.  Moreover, because HS Gibbs is reversible,
we are able to prove that, under weak regularity conditions, adding
sandwich steps always results in an improved algorithm in terms of
both convergence rate and asymptotic variance.  Another advantage that
HS Gibbs has over SS Gibbs (but not over RS Gibbs) is that, if
specific information about the target distribution is known, the
practitioner may vary the selection probability $r \in (0,1)$ to cause
one set of parameters to be updated more frequently than the other.
Lastly, note that the $Z$ component, which is typically used only to
facilitate sampling and is not itself of inferential interest, is
\textit{not} part of the HS Markov chain.  The same is true of the
basic DA algorithm.  While it is possible to marginalize over the $Z$
component in the SS Gibbs chain and still have a bona fide Markov
chain, such marginalization is not possible with the RS Gibbs
algorithm.

It is straightforward to extend the HS Gibbs sampler to situations in
which there are more than three blocks.  Indeed, suppose that breaking
$U$ into two components is not enough.  That is, suppose that we are
unable to identify a partition $u = (x,y)$ such that sampling from
$f_{X|Y,Z}$ and $f_{Y|X,Z}$ is straightforward, but we are able to
find an $s$-component partition, $u=(x_1,x_2,\ldots,x_s)$, such that
it is possible to sample from each $f_{X_j | X_{-j}, Z}$, for
$j=1,\dots,s$, where, as usual,
$X_{-j} = (X_1,\dots,X_{j-1},X_{j+1},\dots,X_s)$.  It is
straightforward to extend the HS algorithm (and all the results that
we discuss in this paper) to this more general case.  For example, at
each iteration of the (generalized) HS algorithm, we update $Z$ and
one randomly chosen element from the random vector $(X_1,\dots,X_s)$.

The only MCMC methods that have been considered so far in this paper
are the DA algorithm and the Gibbs sampler, which could be considered
``classical'' MCMC techniques.  In particular, we have not mentioned
any ``state of the art'' MCMC techniques, such as particle MCMC
\citep{andr:douc:hole:2010} or Hamiltonian Monte Carlo
\citep{neal:2011}.  There are two reasons for this.  Firstly, these
methods are \textit{much} more complex than the classical ones, and
even describing them accurately requires the introduction of a great
deal of notation.  Secondly, and perhaps more importantly, these more
sophisticated methods are often not required to solve a given problem.
Indeed, there are plenty of Bayesian models with posterior
distributions that, while intractable, are perfectly amenable to
classical MCMC methods such as the Gibbs sampler and the
Metropolis-Hastings algorithm.  (Several such examples are detailed in
this paper.)  In such situations, there is no need to consider more
sophisticated MCMC methods, which can be \textit{much} more difficult
to design, code, and analyze than the classical methods.  As an analogy, consider a situation where we have a posterior
distribution that is complex, but from which we can make iid draws (in a
reasonably efficient manner). In such a case, there would be no need to
resort to MCMC since we could effectively explore the posterior using
classical (iid) Monte Carlo.

The remainder of this paper is organized as follows.
Section~\ref{sec:glmm} contains our first serious example of a HS
Gibbs sampler.  The target is the posterior distribution associated
with a Bayesian shrinkage model with random effects.  This algorithm
was first introduced by \citet{abrahamsen2019}, and we restate their
main result, which provides easily checked sufficient conditions for
geometric ergodicity of the underlying Markov chain.  The section ends
with a description of a small empirical study comparing SS, RS and HS
Gibbs.  The topic of Section~\ref{sec:adding} is the \textit{hybrid
  scan sandwich} (HSS) algorithm, which is the result of adding
sandwich steps to a HS Gibbs sampler.  We illustrate the construction
of HSS algorithms by adding sandwich steps to the HS algorithm for our
Student's $t$ example, and to the algorithm described in
Section~\ref{sec:glmm}.  Section~\ref{sec:smn} deals with the
development and analysis of a HS Gibbs sampler for a Bayesian linear
regression model with scale mixtures of normal errors.  A general
result providing sufficient conditions for geometric ergodicity is
stated and applied to two specific mixing densities.  We close with a
discussion in Section~\ref{sec:discussion}.  The Appendix contains
important theoretical results for the general HSS algorithm, as well
as a proof of the convergence rate result stated in
Section~\ref{sec:smn}.

\section{The General Linear Mixed Model with a Continuous Shrinkage Prior}
\label{sec:glmm}

The general linear mixed model takes the form
\begin{equation}
\label{eq:glmm}
Y = X\beta + Zu + e \;,
\end{equation}
where $Y$ is an $N \times 1$ data vector, $X$ and $Z$ are known
matrices with dimensions $N \times p$ and $N \times q$, respectively,
$\beta$ is an unknown $p \times 1$ vector of regression coefficients,
$u$ is a random vector whose elements represent the various levels of
the random factors in the model, $e \sim
\mbox{N}_N(0,\lambda_0^{-1}I),$ and the random vectors $e$ and $u$ are
independent. Suppose that the model contains $m$ random factors, so
that $u$ and $Z$ may be partitioned as $u = (u_1^T \; u_2^T \cdots
u_m^T)^T$ and $Z = (Z_1 \; Z_2 \cdots Z_m),$ where $u_i$ is $q_i
\times 1$, $Z_i$ is $N \times q_i$, and $q_1+ \cdots + q_m = q.$ Then
$Zu = \sum_{i=1}^m Z_i u_i.$ It is assumed that $u \sim
\mbox{N}_q(0,D),$ where $D = \bigoplus_{i=1}^m \lambda_i^{-1}
I_{q_i}.$ Finally, let $\lambda$ denote the vector of precision
components, i.e., $\lambda = (\lambda_0 \; \lambda_1 \cdots
\lambda_m)^T.$

A Bayesian version of the general linear mixed model requires
specification of a prior distribution for the unknown parameters
$\beta$ and $\lambda$.  A popular choice is the proper (conditionally)
conjugate prior that takes $\beta$ to be multivariate normal, and
takes each of the precision components to be gamma. However, in the
increasingly important situation where $p$ is larger than $N,$ we may
wish to use a so-called Bayesian shrinkage prior on $\beta$ (see,
e.g., \citealt{griffin2010inference}). Indeed, \citet{abrahamsen2019}
considered the following Bayesian shrinkage version of the general
linear mixed model which incorporates the normal-gamma prior due to
\citet{griffin2010inference}:
\begin{align*}
& Y|\beta,u,\tau,\lambda \sim \mbox{N}_N(X\beta+Zu,\lambda_0^{-1}I_N)\;\\
& u|\beta,\tau,\lambda \sim \mbox{N}_q(0,D) \; \\
& \beta|\tau,\lambda \sim \mbox{N}_p(0,\lambda_0^{-1}D_{\tau}) \;
\end{align*}
where $D_{\tau}$ is a diagonal matrix with $\tau = (\tau_1 \; \tau_2
\cdots \tau_p)^T$ on the diagonal. Finally, all components of $\tau$
and $\lambda$ are assumed \textit{a priori} independent with
$\lambda_i \sim \text{Gamma}(a_i,b_i),$ for $i=0,1,\ldots,m,$ and
$\tau_j \sim \text{Gamma}(c,d),$ for $j=1,\ldots,p.$ There is evidence
(both empirical and theoretical) suggesting that values of $c$ in
$(0,1/2]$ lead to a posterior that concentrates on sparse $\beta$
  vectors \citep{bhattacharya2012bayesian,bhattacharya2015dirichlet}.

Define $\theta = (\beta^T \; u^T)^T$ and $W = [X \; Z]$, so that
$W\theta = X\beta + Zu$.  The vector $\tau$ is treated as latent data,
and the distribution of interest is the posterior distribution of
$(\theta,\lambda)$ given the data, $Y=y$.  In terms of the notation
used in the Introduction, the role of $Z$ is being played here by
$\tau$, and the role of $(x,y)$ is being played by $(\theta,\lambda)$.
(Ideally, we would keep the notation consistent with that used in the
Introduction, but given how entrenched the roles of $X$, $y$ and $Z$
are in mixed linear regression models, adherence to the notation from
the Introduction would make this section rather difficult to read.)
Here is the full posterior density:
\begin{align}
\begin{split}
\label{eq:fullposterior}
\pi(\theta,\tau,\lambda|y) &\propto \lambda_0^{N/2} \exp \left\{
-\frac{\lambda_0}{2} (y-W\theta)^T (y-W\theta) \right\} \\ &\times
\lambda_0^{p/2} \left[ \prod_{j=1}^p \tau_j^{-1/2} \right] \exp
\left\{ -\frac{\lambda_0}{2} \beta^T D_{\tau}^{-1} \beta \right\}
\left[ \prod_{i=1}^m \lambda_i^{q_i/2} \right] \exp \left\{
-\frac{1}{2} u^T D^{-1} u \right\}\\ &\times \left[ \prod_{j=1}^p
  \tau_j^{c-1} e^{-d\tau_j} I_{\reals_+}(\tau_j) \right] \left[
  \prod_{i=0}^m \lambda_i^{a_i-1} e^{-b_i \lambda_i}
  I_{\reals_+}(\lambda_i) \right] \;.\\
\end{split}
\end{align}
In order to use the basic DA algorithm, we would need to be able to
sample from $\pi(\tau | \theta, \lambda, y)$ and from $\pi(\theta,
\lambda | \tau, y)$.  The former is not a problem, as we now explain.
We write $V \sim \text{GIG}(\zeta,\xi,\psi)$ to mean that $V$ has a
generalized inverse Gaussian distribution with density
\begin{equation}
\label{eq:gig}
\frac{\xi^{\zeta/2}}{2 \psi^{\zeta/2} \; \text{K}_{\zeta}(\sqrt{\xi \psi})} v^{\zeta - 1} e^{-\frac{1}{2}(\xi v + \frac{\psi}{v})}I_{\reals_+}(v),
\end{equation}
where $\xi > 0, \psi > 0,$ and $\text{K}_{\zeta}(\cdot)$ denotes the
modified Bessel function of the second kind. Conditional on $(\theta,
\lambda, y),$ the components of $\tau$ are independent with $$\tau_j
\sim \text{GIG}(c-1/2, 2d, \lambda_0 \beta_j^2).$$ Unfortunately, it
is not straightforward to make draws from $\pi(\theta, \lambda | \tau,
y)$.  Thus, the DA algorithm is not applicable. On the other hand, the
conditional density of $\theta$ given $(\lambda, \tau, y)$ is
multivariate normal, and, given $(\theta, \tau, y),$ the components of
$\lambda$ are independent gammas. Hence, the HS Gibbs algorithm is
applicable.

We now state the conditional densities, beginning with $\lambda.$
First, $$\lambda_0 | \theta, \tau, y \sim \text{Gamma} \left(
\frac{N+p+2a_0}{2}, \frac{\norm{y-W\theta}^2}{2} + \frac{\beta^T
  D_{\tau}^{-1} \beta}{2} + b_0 \right).$$ Now, for $i =
1,2,\ldots,m,$ we have $$\lambda_i | \theta, \tau, y \sim \text{Gamma}
\left( \frac{q_i + 2a_i}{2}, \frac{\norm{u_i}^2}{2} + b_i \right).$$
Now, define $T_{\lambda,\tau} = \lambda_0 (X^T X + D_{\tau}^{-1}),
M_{\lambda,\tau} = I - \lambda_0 X^T T_{\lambda,\tau}^{-1} X^T,$ and
$Q_{\lambda,\tau} = \lambda_0 Z^T M_{\lambda,\tau} Z + D^{-1}.$
Conditional on $(\lambda,\tau,y), \theta$ is $(p+q)$-variate normal
with mean
\begin{equation*}
\Theta := \text{E}[\theta | \tau, \lambda, y]
=
\begin{bmatrix}
\lambda_0 T_{\lambda, \tau}^{-1} X^T y - \lambda_0^2 T_{\lambda,
  \tau}^{-1} X^T Z Q_{\lambda, \tau}^{-1} Z^T M_{\lambda,\tau} y
\\ \lambda_0 Q_{\lambda, \tau}^{-1} Z^T M_{\lambda,\tau} y
\end{bmatrix},
\end{equation*}
and covariance matrix
\begin{equation*}
\Sigma := \text{Var}[\theta | \tau, \lambda, y]
=
\begin{bmatrix}
T_{\lambda, \tau}^{-1} + \lambda_0^2 T_{\lambda, \tau}^{-1} X^T Z
Q_{\lambda, \tau}^{-1} Z^T X T_{\lambda, \tau}^{-1} & -\lambda_0
T_{\lambda, \tau}^{-1} X^T Z Q_{\lambda, \tau}^{-1} \\ -\lambda_0
Q_{\lambda, \tau}^{-1} Z^T X T_{\lambda, \tau}^{-1} & Q_{\lambda,
  \tau}^{-1}
\end{bmatrix}.
\end{equation*}

The HS Gibbs sampler is based on the Markov chain $\Phi = \{(\theta_n,
\lambda_n)\}_{n=0}^{\infty}$ with state space $\mathsf{X} =
\reals^{p+q} \times \reals_+^{m+1}$ and fixed selection probability $r
\in (0,1)$.  If the current state is $(\theta_n,
\lambda_n)=(\theta,\lambda),$ then we simulate the new state,
$(\theta_{n+1}, \lambda_{n+1}),$ using the following two-step
procedure.

\baro \vspace*{2mm}
\noindent {\rm Iteration $n+1$ of the HS Gibbs sampler:}
\begin{enumerate}
\item Draw $\{\tau_j\}_{j=1}^p$ independently with $\tau_j \sim
  \text{GIG}(c-1/2, 2d, \lambda_0 \beta_j^2)$, let $\tau = (\tau_1 \;
  \tau_2 \cdots \tau_p)^T,$ and, independently, draw $W \sim
  \text{Uniform}(0,1).$

\item 
\begin{enumerate}
\item If $W \le r,$ draw $(\lambda^*_0,\ldots,\lambda^*_m)$
  independently with $$\lambda^*_0 \sim \text{Gamma} \left(
  \frac{N+p+2a_0}{2}, \frac{\norm{y-W\theta}^2}{2} + \frac{\beta^T
    D_{\tau}^{-1} \beta}{2} + b_0 \right),$$ and for
  $i=1,\dots,m$, $$\lambda^*_i \sim \text{Gamma} \left( \frac{q_i +
    2a_i}{2}, \frac{\norm{u_i}^2}{2} + b_i \right),$$ and let
  $\lambda^* = (\lambda^*_0 \; \lambda^*_1 \cdots \lambda^*_m)^T.$ Set
  $(\theta_{n+1}, \lambda_{n+1}) = (\theta, \lambda^*)$.

\item Otherwise if $r < W \le 1,$ draw $$\theta^* \sim \mbox{N}_{p+q}
  \left( \Theta, \Sigma \right),$$ and set $(\theta_{n+1},
  \lambda_{n+1}) = (\theta^*, \lambda).$
\end{enumerate}
\end{enumerate}
\vspace*{-3mm}
\barba
\bigskip

We know from the discussion in the Introduction that the Markov chain
driving this algorithm is reversible with respect to
$\pi(\theta,\lambda|y)$.  Furthermore, it is straightforward to show
that this chain is Harris ergodic (i.e., irreducible, aperiodic and
Harris recurrent). \citet{abrahamsen2019} analyzed this HS Gibbs
sampler, and proved that it is geometrically ergodic under mild
regularity conditions.  Here is their main result.

\begin{theorem}
\label{thm:glmm}
The HS Gibbs Markov chain, $\{(\theta_n,\lambda_n)\}_{n=0}^{\infty}$
is geometrically ergodic for all $r \in (0,1)$ if
\begin{enumerate}
\item $Z := (Z_1 \; Z_2 \; \cdots \; Z_m)$ has full column rank.
\item $a_0 > \frac{1}{2} \big( \hbox{rank}(X) - N + (2c+1)p+2 \big)$,
  and
\item $a_i > 1$ for each $i \in \{1,2,\ldots,m\}$.
\end{enumerate}
\end{theorem}

The conditions of Theorem~\ref{thm:glmm} are quite easy to check, and
the result is applicable when $p > N$.  Moreover, there are no known
convergence rate results for the corresponding SS and RS Gibbs
samplers.  Indeed, \citet{abrahamsen2019} contend that HS Gibbs is
much easier to analyze than the other two, despite being no more
difficult to implement.

We now compare the HS, SS, and RS Gibbs samplers in the context of a
numerical example.  (No numerical results were presented in
\citet{abrahamsen2019}.)  We also include in the comparison the hybrid
scan sandwich algorithm, which is developed in
Subsection~\ref{ssec:glmm}.  We consider three simulation settings
corresponding to the situations where $N > p$, $N = p$, and $N < p$,
respectively, in order to account for the effects of the shrinkage
prior.  The elements of the design matrix $X$ were chosen by
generating iid $N(0,1)$ random variables.  There is one random effect
with 5 levels, i.e., $m=1$ and $q_1=q=5$, and we use the standard cell
means model structure for the matrix $Z$.  Recall from
Theorem~\ref{thm:glmm} that there are several restrictions on the
hyperparameters that must be adhered to in order for the HS Gibbs
Markov chain to be geometrically ergodic. This sometimes requires
$a_0$ to be large. We mitigate this by setting $b_0 = a_0$ in each
simulation setting to give the corresponding prior distribution a mean
of 1.  We set $a_1 = 1.5$ and $b_1 = 1$ for all three simulations.
Also, recall that there is empirical and theoretical evidence
suggesting that values of $c$ in $(0,1/2]$ lead to a posterior that
concentrates on sparse $\beta$ vectors.  Accordingly, we set $c=1/4$
and $d=1$ throughout.  Here is a summary of the simulation settings
considered.

\begin{center}
Table 1:
\textit{Hyperparameter settings}
\end{center}
\begin{center}
\begin{tabular}{ccccccccccc}
\hline
\textbf{Setting} & $\boldsymbol{N}$ & $\boldsymbol{p}$ & $\boldsymbol{m}$ & $\boldsymbol{q}$ & $\boldsymbol{a_0}$ & $\boldsymbol{b_0}$ & $\boldsymbol{a_1}$ & $\boldsymbol{b_1}$ & $\boldsymbol{c}$  & $\boldsymbol{d}$ \\ \hline
1       & 100 & 10  & 1   & 5   & 1     & 1     & 1.5   & 1     & 0.25 & 1   \\
2       & 100 & 100 & 1   & 5   & 77    & 77    & 1.5   & 1     & 0.25 & 1   \\
3       & 100 & 200 & 1   & 5   & 152   & 152   & 1.5   & 1     & 0.25 & 1   \\ \hline
\end{tabular}
\end{center}
\medskip

\noindent In each case, the data were simulated according to the model
using a ``bottom up'' strategy, i.e., the hyperparameters were
randomly drawn from their priors, and so on, up the hierarchy.

We fix the selection probability at $r = 1/2$ for the HS and HSS
algorithms.  For RS Gibbs, we fix the selection probabilities at
$r_1=r_2=r_3=1/3$.  We wish to compare the algorithms using
autocorrelation plots, but the four algorithms make different numbers
of updates per iteration.  Indeed, the SS, HS, HSS and RS algorithms
make 3, 2, 2 and 1 updates/iteration, respectively.  So an adjustment
must be made in order to perform an ``apples to apples'' comparison.
If $k$ is a positive integer, then it seems fair to compare the lag
$2k$ autocorrelation for SS algorithm with the lag $3k$
autocorrelation for the HS and HSS algorithms, and the lag $6k$
autocorrelation for the RS algorithm.

In each of the three separate simulations, we ran the SS, HS, HSS, and
RS algorithms for 40,000 iterations, 60,000 iterations, 60,000
iterations, and 120,000 iterations, respectively.  We then discarded
the first half of each run as burn-in, and computed the
autocorrelations based on the remaining data as described above.  We
used the function $(y - W \theta)^T(y - W \theta) + \lambda_0 +
\lambda_1$ because it involves both parameters of interest ($\theta$
and $\lambda$).  The results are summarized in
Figure~\ref{fig:simplot}.  (Just to be clear, as an example, what is
plotted above the abscissa at the value 6 for the SS, HS, HSS, and RS
algorithms is the estimated autocorrelation for lag 6, 9, 9, and 18,
respectively.)  We can clearly see that for all three simulations, the
magnitude of the autocorrelations for SS Gibbs is the lowest, while
the other three are all a bit higher and quite close to each other.
The performances of the HS Gibbs sampler and the HSS algorithm are
similar.

\begin{figure}[H]
  \includegraphics[width=\linewidth]{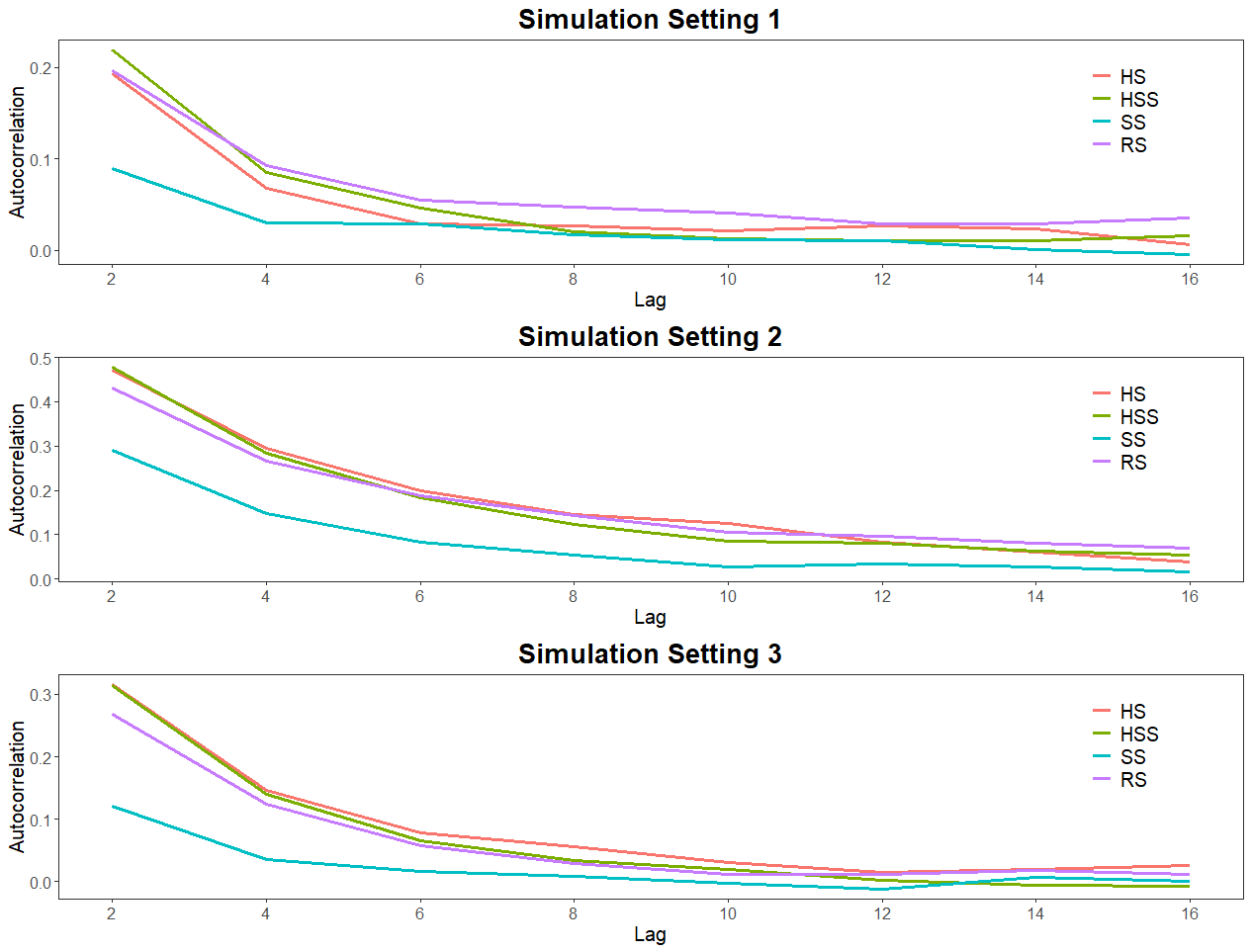}
  \caption{Autocorrelations for SS, HS, HSS, and RS algorithms}
  \label{fig:simplot}
\end{figure}

While it is true that the SS Gibbs algorithm seems to be marginally
better than the others in this particular case, recall that it remains
unknown whether the corresponding Markov chain is geometrically
ergodic.  On the other hand, the HS chain (and the HSS chain - see
Subsection~\ref{ssec:HSS}) are both known to be geometrically ergodic.
Thus, in order to ensure reliability, we recommend the two ``known
quantities.''

\section{The Hybrid Scan Sandwich Algorithm}
\label{sec:adding}

In this section, we explain how to add sandwich steps to the HS Gibbs
sampler to form the \textit{hybrid scan sandwich} (HSS) algorithm.
There are four subsections.  The basic sandwich algorithm of
\citet{hobert2008theoretical} is described in
Subsection~\ref{ssec:basics}.  A generic description of the HSS
algorithm is provided in Subsection~\ref{ssec:HSS}.  In
Subsection~\ref{ssec:toy}, we illustrate the techniques using the
Student's $t$ example from the Introduction.  Finally, in
Subsection~\ref{ssec:glmm}, we develop a HSS algorithm for the
intractable posterior associated with the mixed model discussed in
Section~\ref{sec:glmm}.

\subsection{The sandwich algorithm}
\label{ssec:basics}

In keeping with the notation of the Introduction, recall that the
transition associated with a single iteration of the DA algorithm may
be represented as
\[
\begin{tikzcd}
U \arrow[r] & Z \arrow[r] & U' \;.
\end{tikzcd}
\]
Building on ideas in \citet{liu1999parameter}, \citet{meng1999seeking}
and \citet{van2001art}, \citet{hobert2008theoretical} introduced an
alternative to the DA algorithm that employs an extra move on the
$\mathsf{Z}$ space that is ``sandwiched'' between the two conditional
draws.  If the extra move is chosen carefully, it can break the
correlation between consecutive iterates of the DA algorithm, thereby
speeding up the algorithm.  Again, using notation from the
Introduction, let $f_Z$ denote the $z$-marginal of $f(u,z)$, and
suppose that $R(z,dz')$ is \textit{any} Markov transition function
(Mtf) that is reversible with respect to $f_Z$, i.e., $R(z,dz') f_Z(z)
dz = R(z',dz) f_Z(z') dz'$.  The \textit{sandwich algorithm} simulates
the Markov chain whose Mtd is
\[
k_{\text{S}}(u' | u) = \int_{\mathsf{Z}} \int_{\mathsf{Z}}
f_{U|Z}(u'|z') R(z,dz') f_{Z|U}(z|u) \, dz \;.
\]
It's easy to see that $k_{\text{S}}(u' | u) f_U(u)$ is symmetric in
$(u,u')$, so the sandwich Markov chain is reversible with respect to
$f_U$.  Also, the sandwich algorithm reduces to DA if we take $R$ to
be the trivial Mtf whose chain never moves from the starting point.
To run the sandwich algorithm, we simply run the DA algorithm as
usual, except that after each $z$ is drawn, we perform the extra step
$z' \sim R(z,\cdot)$ before drawing the new $u$.  Hence, the sequence
of steps in a single iteration of the sandwich algorithm looks like
this:
\[
\begin{tikzcd}
U \arrow[r] & Z \arrow[r] & Z' \arrow[r] & U' \;.
\end{tikzcd}
\]
We now explain how a sandwich step can effectively break the
correlation between $U$ and $U'$ in the context of a toy example.

Suppose the target density is
\[
f_U(u) = \int_{\mathbb{R}} \frac{1}{\sqrt{8\pi}} e^{-\frac{1}{2}
  (u-z)^2 -|z|} \; dz \;.
\]
In order to construct a DA algorithm, we require a joint density whose
$u$-marginal is the target.  Here's an obvious candidate:
\[
f_{U,Z}(u,z) = \frac{1}{\sqrt{8\pi}} e^{-\frac{1}{2} (u-z)^2 -|z|} \;.
\]
Note that $f_Z(z) = \frac{1}{2} e^{-|z|}$, so the marginal
distribution of $Z$ is standard Laplace (or double exponential).  In
order to run the DA algorithm, we need the full conditionals.
Clearly, $U|Z \sim \mbox{N}(Z,1)$, but the distribution of $Z$ given
$U$ is non-standard:
\[
f_{Z|U}(z|u) \propto e^{-\frac{1}{2} (u-z)^2 -|z|} \;.
\]
It's a simple matter to simulate from this density using a rejection
sampler with a Laplace candidate.  We now construct a sandwich
algorithm.  Define a Mtf on $\mathbb{R}$ as follows:
\[
R(z,dz') = r(z'|z) dz' = e^{-|z'|} \big[ I_{\mathbb{R}_+}(z')
  I_{\mathbb{R}_+} (z) + I_{\mathbb{R}_-}(z') I_{\mathbb{R}_-} (z)
  \big] dz' \;,
\]
where $\mathbb{R}_- := (-\infty,0]$.  It's clear that $R(z,dz')$ is
  reversible with respect to $f_Z(z)$.  Note that the Markov chain
  defined by $R$ is \textit{not} irreducible.  In fact, the chain
  remains forever on whichever side of zero it is started.  We now
  provide some intuition about how the extra step breaks the
  correlation between $U$ and $U'$.  Imagine for a moment that
  $r(z'|z)$ were just $f_Z(z')$.  Then $U'$ would be a perfect draw
  from $f_U$ (independent of $U$), and the Markov chain would simply
  be an iid sequence from the target distribution.  Of course,
  $r(z'|z)$ is not $f_Z(z')$, but it actually isn't that far from it.
  First, $r(z'|z)$ depends on $z$ \textit{only through its sign}.
  Now, when $z>0$, $r(z'|z)$ is nothing but $f_Z(z')$ truncated to the
  positive half-line, and when $z \le 0$, $r(z'|z)$ is just $f_Z(z')$
  truncated to the negative half-line.  So we can interpret the extra
  step as follows: Once $Z$ is drawn from $f_{Z|U}(\cdot|u)$, the
  extra step then draws $Z'$ from a truncated version of $f_Z$.
  Intuitively, it seems clear that the correlation between $U$ and
  $U'$ should be quite a bit weaker under the sandwich dynamics, than
  under the DA dynamics.  In order to test this empirically, we ran
  stationary versions of each chain for one million iterations, and
  constructed the autocorrelation plot in Figure~\ref{fig:toyacf}
  using the function $V(u)=u^2$.  Clearly, the autocorrelation of the
  sandwich Markov chain decays to zero much more rapidly than that of
  the corresponding DA chain.
  
\begin{figure}[H]
  \includegraphics[width=\linewidth]{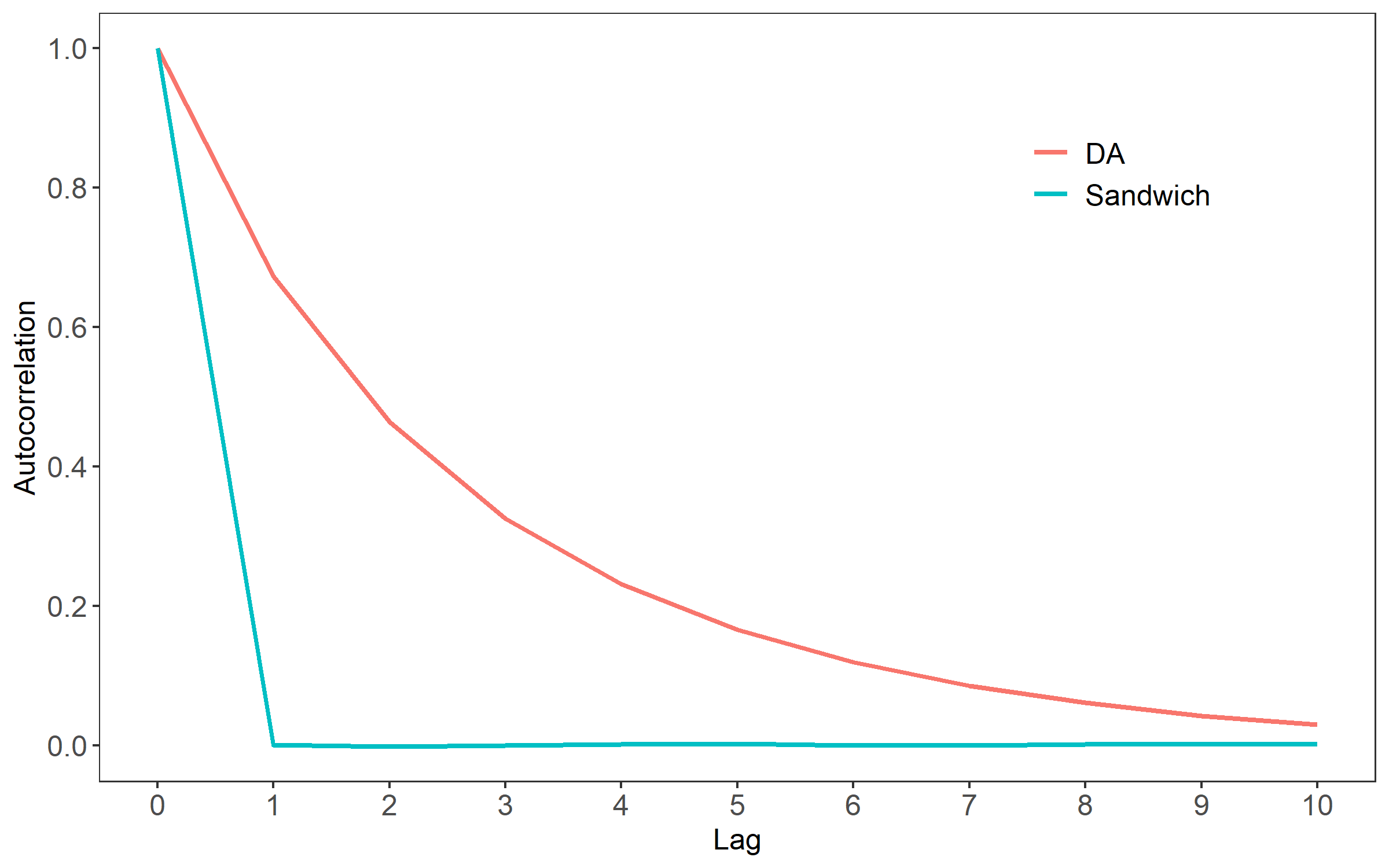}
  \caption{Autocorrelations for the DA and sandwich algorithms}
  \label{fig:toyacf}
\end{figure}

Of course, a sandwich algorithm is a useful alternative to the
underlying DA algorithm only if the computational burden of drawing
from $R$ is small relative to the improvement it provides.  Consider,
for example, the Mtf $R(z,dz') = r(z'|z)dz'$ where $r(z'|z) =
\int_{\mathsf{U}} f_{Z|U}(z'|u) f_{U|Z}(u|z)du$.  This $R$ leads to a
sandwich algorithm that is nothing but two consecutive iterations of
the DA algorithm.  Thus, whatever is gained by adding the extra step
is offset exactly in increased computational effort.  Fortunately, it
is often possible to find an $R$ that leads to a significant
improvement, while adding very little to the overall computational
cost. This is typically accomplished by choosing $R(z,dz')$ such that,
for fixed $z$, the (reducible) chain driven by $R(z,\cdot)$ lives in a
low dimensional subspace of $\mathsf{Z}.$ (Note that such an $R$ would
typically not have a Mtd with respect to Lebesgue measure on
$\mathsf{Z}$, and this is the reason why it is defined via its Mtf,
instead of a Mtd.)

There are a couple of simple techniques for constructing sandwich
moves \citep[see, e.g.,][]{hobert2008theoretical,liu1999parameter},
and the resulting Mtfs can often be simulated with relatively little
computational effort.  In such cases, there is nothing to lose by
adding the step.  In other cases, where simulation of the extra step
requires substantial computational effort, one must decide if the
trade-off is worthwhile.  There are many examples of sandwich
algorithms that drastically outperform their DA counterparts in
empirical studies, see, e.g., \citet{liu1999parameter} and
\citet{meng1999seeking}.  Moreover, the superiority of the sandwich
algorithm has also been established theoretically.  Indeed, results in
\citet{hobert2008theoretical} and \citet{khare2011spectral} show that,
under mild regularity conditions, the sandwich algorithm converges at
least as fast as the DA algorithm, and is at least as good in the
sense of asymptotic variance.

\subsection{The HSS algorithm}
\label{ssec:HSS}

We now explain how to add up to two different sandwich steps to the HS
Gibbs sampler.  Recall that the transition mechanism for each iteration of HS Gibbs with selection probability $r$ is given by 
\[
\begin{tikzcd}
& & (X^*,Y) \\
(X,Y) \arrow[r] & Z \arrow[ur,"r"] \arrow[dr,"1-r"] & \;\\
& & (X,Y^*)
\end{tikzcd}
\]
For fixed $y \in \mathsf{Y},$ let $R_1(z,dz';y)$
denote a Mtf on $\mathsf{Z}$ that is reversible with respect to
$f_{Z|Y}(z|y),$ so that
\begin{equation}
\label{eq:reverse}
R_1(z,dz';y) f_{Z|Y}(z|y) dz = R_1(z',dz;y) f_{Z|Y}(z'|y) dz' \;.
\end{equation}
Define 
\[
k_1(x'|x;y) = \int_{\mathsf{Z}} \int_{\mathsf{Z}} f_{X|Y,Z}(x'|y,z')
R_1(z,dz';y) f_{Z|X,Y}(z|x,y) \, dz \;.  
\]
A routine calculation shows that $k_1(x'|x;y)f_{X|Y}(x|y)$ is
symmetric in $(x,x')$, so $k_1(x'|x;y)$ is reversible with respect to
$f_{X|Y}(x|y)$.  Analogously, for fixed $x \in \mathsf{X}$,
define
\[
k_2(y'|y;x) = \int_{\mathsf{Z}} \int_{\mathsf{Z}} f_{Y|X,Z}(y'|x,z')
R_2(z,dz';x) f_{Z|X,Y}(z|x,y) \, dz \;,
\]
where $R_2(z,dz';x)$ is reversible with respect to $f_{Z|X}(z|x)$.

The HSS algorithm is simply a RS algorithm which, at each iteration,
employs either $k_1(x'|x;y)$ or $k_2(y'|y;x)$.  In particular, fix $r
\in (0,1)$, and consider a Markov chain $\{(\tilde{X}_n,\tilde{Y}_n)
\}_{n=0}^{\infty}$ with state space $\mathsf{X} \times \mathsf{Y}$
that evolves as follows.  If the current state is
$(\tilde{X}_n,\tilde{Y}_n) = (x,y)$, then we simulate the new state,
$(\tilde{X}_{n+1},\tilde{Y}_{n+1})$, using the following two-step
procedure.

\baro \vspace*{2mm}
\noindent {\rm Iteration $n+1$ of the HSS algorithm:}
\begin{enumerate}
\item Draw $Z \sim f_{Z|X,Y}(\cdot|x,y)$, call the result $z$, and,
  independently, draw $W \sim \text{Uniform}(0,1)$.

\item 
\begin{enumerate}
\item If $W \le r$, draw $Z' \sim R_1(z,\cdot;y)$, call the result
  $z'$, draw $X^* \sim f_{X|Y,Z}(\cdot|y,z')$, and set
  $(\tilde{X}_{n+1},\tilde{Y}_{n+1}) = (X^*,y)$.

\item Otherwise if $r < W \le 1$, draw $Z' \sim R_2(z,\cdot;x)$, call
  the result $z'$, draw $Y^* \sim f_{Y|X,Z}(\cdot|x,z')$, and set
  $(\tilde{X}_{n+1},\tilde{Y}_{n+1}) = (x,Y^*)$.
\end{enumerate}
\end{enumerate}
\vspace*{-3mm}
\barba
\bigskip

\newpage
Thus, the HSS algorithm makes the following transition at each
iteration.
\[
\begin{tikzcd}
						                  & & Z' \arrow[r] & (X^*,Y) \\
(X,Y) \arrow[r] & Z \arrow[ur,"r"] \arrow[dr,"1-r"] & \\
						                  & & Z' \arrow[r] & (X,Y^*)
\end{tikzcd}
\]
If we take both $R_1$ and $R_2$ to be trivial, then the HSS algorithm
collapses back into the HS Gibbs sampler.  In
Appendix~\ref{app:theory}, we develop theoretical results for the HSS
algorithm.  We begin by showing that the HSS algorithm is reversible,
which allows us to prove analogues for the HSS algorithm of the strong
theoretical results that have been established for the basic sandwich
algorithm.  In particular, we prove that the HSS algorithm is
\textit{always} at least as good as HS Gibbs in terms of asymptotic
variance, and that the HSS Markov chain converges at least as fast as
the HS Gibbs chain as long as the Markov operators associated with
$R_1$ and $R_2$ are both \textit{positive}.  (All of the $R$s employed
in this paper, including the trivial $R$, yield positive Markov
operators - see \citet{hobert2008theoretical}.)  One important consequence of the convergence rate result is that, when
$R_1$ and $R_2$ are both positive operators, geometric ergodicity of the
HS Gibbs Markov chain implies that of the HSS chain.

This is extremely useful in practice because the HS Gibbs algorithm is
much simpler than the HSS algorithm, and hence much easier to analyze.

We should point out that \citet{pal2015improving} also developed an
alternative to SS and RS Gibbs for Bayesian latent data models that is
based on the sandwich methodology of \citet{hobert2008theoretical}.
Unfortunately, it is difficult to obtain theoretical results for their
algorithm because the corresponding Markov chain is not reversible.

Recall that near the end of the Introduction we considered a
generalization in which $U$ is partitioned into three or more pieces,
and we wrote the corresponding conditional densities as $f_{X_j |
  X_{-j}, Z}$ for $j=1,\dots,s$.  It is a simple matter to extend the
methodology described above to this more general case.  Indeed, for
fixed $X_{-j} = x_{-j}$, let $R_j(z,dz';x_{-j})$ denote a Mtf on
$\mathsf{Z}$ that is reversible with respect to $f_{Z |
  X_{-j}}(z|x_{-j})$.  Define
\[
k_j(x'_j | x_j; x_{-j}) = \int_{\mathsf{Z}} f_{X_j | X_{-j},
  Z}(x'|x_{-j},z') R_j(z,dz';x_{-j}) f_{Z|X_j, X_{-j}}(z|x_j,x_{-j})
\, dz \;.
\]
At each step of the generalized version of HSS, we choose among
$k_1,\ldots,k_s$ according to positive probabilities $a_1,\dots,a_s$
in the usual way, and apply the chosen $k_j$.  All of the theoretical
results that we establish for the HSS algorithm in
Appendix~\ref{app:theory} can be easily extended to this
generalization.

\subsection{Student's $t$ example}
\label{ssec:toy}

Consider again the first Student's $t$ model from the Introduction
(with prior $\pi(\mu,\sigma^2) \propto \sigma^{-2}
I_{\mathbb{R}^+}(\sigma^2)$).  We now develop a HSS algorithm for this
model.  Of course, we already know that this model can be handled by
the usual DA algorithm, so our HSS algorithm would never be used in
practice.  However, we believe that it is instructive to demonstrate
the construction of a HSS algorithm in a simple context where the
details of the model are not themselves overwhelming.

The first step is to identify the distributions of $Z | \mu$ and $Z |
\sigma^2$.  Let $z_{\cdot} = \sum_{i=1}^m z_i$.  It's easy to show
that
\begin{align*}
f_{Z|\mu}(z|\mu) & \propto \bigg( \sum_{i=1}^m z_i (w_i - \mu)^2
\bigg)^{-\frac{m}{2}} \Bigg[ \prod_{i=1}^m z_i \Bigg]^{\frac{\nu -
    1}{2}} \exp \bigg \{ \! -\frac{z_{\cdot} \nu}{2} \bigg \}
\prod_{i=1}^m I_{\mathbb{R}^+}(z_i) \;.
\end{align*}
Let $g \in \reals_+$.  It follows from the group theoretic arguments
in \citet{hobert2008theoretical} that the move $z \mapsto g z$ for
$z=(z_1,\ldots,z_m)$ is reversible with respect to $f_{Z|\mu}(z|\mu)$
if $g$ is drawn from the density proportional to $f_{Z|\mu}(gz | \mu)
g^{m-1}$.  (This is a low-dimensional move since, for fixed $z \in
\reals_+^m$, the points $gz$ all lie on a ray emanating from the
origin and passing through the point $z$.)  As a function of $g$, we
have
\begin{align*}
f_{Z|\mu}(gz | \mu) g^{m-1} & \propto g^{\frac{m \nu}{2} - 1} \exp
\bigg \{ \! -\frac{g \nu z_{\cdot}}{2} \bigg \} \Bigg[ \prod_{i=1}^m
  I_{\mathbb{R}^+}(z_i) \Bigg] I_{\mathbb{R}^+}(g) \;,
\end{align*}
which is a $\text{Gamma}\big( \frac{m \nu}{2}, \frac{\nu z_{\cdot}}{2}
\big)$ density.  Now, it's easy to show that
\begin{align*}
f_{Z|\sigma^2}(z|\sigma^2) & \propto \frac{1}{\sqrt{z_{\cdot}}} \exp
\bigg \lbrace \! -\frac{z_{\cdot} v(z,w)}{2\sigma^2} \bigg \rbrace
\Bigg[ \prod_{i=1}^m z_i \Bigg]^{\frac{\nu - 1}{2}} \exp \bigg \{ \!
-\frac{z_{\cdot} \nu}{2} \bigg \} \prod_{i=1}^m I_{\mathbb{R}^+}(z_i)
\;,
\end{align*}
where
\[
v(z,w) = \frac{1}{z_{\cdot}} \sum_{i=1}^m z_i \big( w_i - \theta(z,w)
\big)^2 \;,
\]
and $\theta(z,w) = \frac{1}{z_{\cdot}} \sum_{i=1}^m z_i w_i$.  Using
the same transformation, $z \mapsto g z$, we need to sample $g$ from
the density proportional to $f_{Z|\sigma^2}(gz | \sigma^2) g^{m-1}$.
A straightforward calculation shows that, as a function of $g$, we
have
\begin{align*}
f_{Z|\sigma^2}(gz | \sigma^2) g^{m-1} & \propto g^{\frac{m(\nu
    +1)-3}{2}} \exp \left\lbrace -g z_. \bigg(
\frac{v(z,w)}{2\sigma^2} + \frac{\nu}{2} \bigg) \right \rbrace \;,
\end{align*}
which is a $\text{Gamma} \Big( \frac{m(\nu + 1) - 1}{2} , z_. \big(
\frac{v(z,w)}{2\sigma^2} + \frac{\nu}{2} \big) \Big)$ density.

Fix a selection probability $r \in (0,1)$ and consider the Markov
chain $\{(\tilde{\mu}_n,\tilde{\sigma}^2_n)\}_{n=0}^{\infty}$ with
state space $\reals \times \reals_+$.  The HSS algorithm proceeds as
follows.  If the current state is $(\tilde{\mu}_n,\tilde{\sigma}^2_n)
= (\mu,\sigma^2)$, then we simulate the next state,
$(\tilde{\mu}_{n+1},\tilde{\sigma}^2_{n+1}),$ by performing the
following two steps:

\baro \vspace*{2mm}
\noindent {\rm Iteration $n+1$ of the HSS algorithm for the Student's
  $t$ example:}
\begin{enumerate}
\item Draw $Z_1, \ldots, Z_m$ independently, with $$Z_i \sim
  \text{Gamma} \left( \frac{\nu + 1}{2}, \frac{1}{2} \left( \frac{(w_i
    - \mu)^2}{\sigma^2} + \nu \right) \right) \;,$$ call the observed
  values $z=(z_1,\ldots,z_m),$ and, independently, draw $W \sim
  \text{Uniform}(0,1).$

\item
\begin{enumerate}
\item If $W \le r,$ draw $$g \sim \text{Gamma} \bigg( \frac{m \nu}{2},
  \frac{\nu z_{\cdot}}{2} \bigg) \,,$$ then draw $$\sigma^{*2} \sim
  \text{IG} \bigg( \frac{m}{2}, \frac{1}{2} \sum_{i=1}^m g z_i (y_i -
  \mu)^2 \bigg),$$ and set $(\tilde{\mu}_{n+1},\tilde{\sigma}^2_{n+1})
  = (\mu,\sigma^{*2}).$

\item Otherwise if $r < W \le 1,$ draw $$g \sim \text{Gamma} \bigg(
  \frac{m(\nu + 1) - 1}{2} , z_. \Big( \frac{v(z,w)}{2\sigma^2} +
  \frac{\nu}{2} \Big) \bigg) \;,$$ and then draw $$\mu^* \sim \text{N}
  \left( \theta(z,w), \frac{\sigma^2}{g z_.} \right),$$ and set
  $(\tilde{\mu}_{n+1},\tilde{\sigma}^2_{n+1}) = (\mu^*,\sigma^2)$.
\end{enumerate}
\end{enumerate}
\vspace*{-3mm}
\barba
\bigskip

In terms of computation, the difference between running one iteration
of this HSS algorithm versus one iteration of the HS Gibbs sampler
upon which it is based is a single draw from the gamma distribution.
Thus, if $m$ is even moderately large, this extra draw would add
relatively little to the overall computational effort of the HS Gibbs
algorithm.

\subsection{General linear mixed model example}
\label{ssec:glmm}

\citet{abrahamsen2019} introduced and analyzed the HS Gibbs sampler
described in Section~\ref{sec:glmm}, but they did not consider adding
sandwich steps to their algorithm.  In this subsection, we develop a
HSS algorithm with a single sandwich step based on the conditional
density $\pi(\tau|\theta,y)$.  (It is much more difficult to construct
a sandwich step based on $\pi(\tau|\lambda,y)$.)  A routine
calculation shows that
\[
\pi(\tau|\theta,y) \propto \left( \frac{\norm{y-W\theta}^2}{2} +
\frac{\beta^T D_{\tau}^{-1} \beta}{2} + b_0 \right)^ {-
  \left(\frac{N}{2}+\frac{p}{2}+a_0 \right)} \prod_{j=1}^p
\tau_j^{c-\frac{3}{2}} e^{-d \tau_j} I_{\reals_+}(\tau_j) \,.
\]
As in the previous subsection, the move $\tau \mapsto g \tau$ is
reversible with respect to $\pi(\tau|\theta,y)$ if $g$ is drawn from
the density proportional to $\pi(g \tau|\theta,y) g^{p-1}
I_{\reals_+}(g)$.  Now, as a function of $g$,
\[
\pi(g \tau | \theta, y) \propto \left( \frac{\norm{y-W\theta}^2}{2} +
\frac{g^{-1} \beta^T D_{\tau}^{-1} \beta}{2} + b_0 \right)^ {-
  \left(\frac{N}{2}+\frac{p}{2}+a_0 \right)} g^{p(c-\frac{3}{2})}
e^{-g \left( d \sum_{j=1}^p \tau_j \right)} \prod_{j=1}^p
I_{\reals_+}(\tau_j) \;,
\]
so the density from which $g$ must be drawn is given by
\[
h(g ; \tau, \theta, y) \propto
\frac{g^{\frac{N}{2}+cp+a_0-1-s}}{\left( \frac{\beta^T D_{\tau}^{-1}
    \beta}{2} + g \left( \frac{\norm{y-W\theta}^2}{2} + b_0 \right)
  \right)^{\frac{N}{2}+\frac{p}{2}+a_0}} \; \left[ g^s e^{-g \left( d
    \sum_{j=1}^p \tau_j \right)} \right] I_{\reals_+}(g) \;,
\]
where $s > 0$ is a free parameter.  So,
\begin{equation}
\label{eq:hg}
h(g ; \tau, \theta, y) \propto \frac{g^{\frac{N}{2}+cp+a_0-1-s}}{(1 +
  C g)^{\frac{N}{2}+\frac{p}{2}+a_0}} \; \left[ g^s e^{-g \left( d
    \sum_{j=1}^p \tau_j \right)} \right] I_{\reals_+}(g) \;,
\end{equation}
where 
\[
C = \frac{\norm{y-W\theta}^2 + 2 b_0}{\beta^T D_{\tau}^{-1}\beta} \;.
\]
If we choose $s \in \Big( \max \Big\{ 0, \; p \big(c - \frac{1}{2}
\big) \Big \} , \frac{N}{2} + cp + a_0 \Big)$, then two things happen:
(1) the first term on the right-hand side of ~\eqref{eq:hg} is
proportional to a scaled $F$ density, and (2) the second term is
bounded. In fact, the second term achieves its maximum at $\hat{g} = s
\big( d \sum_{j=1}^p \tau_j \big)^{-1}$.  Thus, we can use a simple
accept/reject algorithm with an $F$ candidate to draw from $h.$ In
particular, let $\nu_1 = N + 2cp + 2a_0 -2s$ and $\nu_2 = p(1-2c) +
2s.$ Here's the algorithm.

\baro \vspace*{2mm}
\noindent {\rm Accept/Reject algorithm for $h$:}
\begin{enumerate}
\item Draw $V^* \sim F(\nu_1,\nu_2),$ set $V = (V^* \nu_1) / (C
  \nu_2),$ and independently draw $U \sim \text{Uniform}(0,1).$

\item If $$U \le \left( \frac{d V \sum_{j=1}^p \tau_j}{s} \right)^s \;
  e^{s - d V \sum_{j=1}^p \tau_j},$$ then accept $V$ as a draw from
  ~\eqref{eq:hg}, otherwise return to 1.
\end{enumerate}
\vspace*{-3mm}
\barba
\bigskip

If $r \in (0,1)$ is the selection probability, then our HSS algorithm
proceeds as follows.  Let the current state of the chain be
$(\theta_n, \lambda_n) = (\theta, \lambda)$.  First, draw $\tau \sim
\pi(\tau | \theta, \lambda, y)$, and then flip an $r$-coin.  If the
coin comes up heads, we move to $(\theta_{n+1}, \lambda_{n+1}) =
(\theta, \lambda^*)$ by first drawing $g \sim h(\cdot ; \tau, \theta,
y)$ and then drawing $\lambda^* \sim \pi(\lambda| \theta, g \tau, y)$.
If the coin comes up tails, we move to $(\theta_{n+1}, \lambda_{n+1})
= (\theta^*, \lambda)$ by drawing $\theta^* \sim \pi(\theta | \lambda,
\tau, y)$.  Another, perhaps simpler, way to describe the HSS
algorithm is via a simple modification of the HS Gibbs algorithm described
in Section~\ref{sec:glmm}.  Step 1 remains exactly the same.  In step
2, if $r < W \le 1$, then, again, nothing changes.  However, if $W \le
r$, then, instead of using $\tau$ from step 1, we draw $g \sim h(\cdot
; \tau, \theta, y)$, and use $g \tau$ in place of $\tau$.

It follows from Proposition~\ref{prop:main_comp} in
Appendix~\ref{app:theory} that, whenever the HS Gibbs sampler of
Section~\ref{sec:glmm} is geometrically ergodic, so is our HSS
algorithm.  Recall that some empirical results for this HSS algorithm
are depicted alongside the results for the HS, SS, and RS Gibbs
samplers in Figure~\ref{fig:simplot} of Section~\ref{sec:glmm}.  In
that example, the rejection sampler is quite efficient, with an
acceptance probability of more than 70\% in each of the three
simulations settings considered.  The per iteration computational cost
of HS Gibbs obviously grows with $p$ while the extra cost associated
with rejection sampling is basically constant in $p$.  As a result, in
the second and third simulation settings, the HSS algorithm was only
about 2\% slower than HS Gibbs, while in the first setting, the HSS
algorithm is substantially slower than HS Gibbs.  Note that the
performance of the rejection sampler is a function of $C$ and
$\sum_{j=1}^p \tau_j$. For these simulations, we developed a table in
a preliminary offline investigation to decide the appropriate value of
the free parameter $s$ for a given $\big( C,\sum_{j=1}^p \tau_j \big)$
pair.

\section{Bayesian Linear Regression with Scale Mixtures of Normal Errors}
\label{sec:smn}

In this section, we provide another example of a Bayesian model that
leads to a highly intractable posterior distribution that lends itself
to the HS Gibbs sampler.  Let $Y_1,\ldots,Y_m$ be independent random
variables from the linear regression model
\begin{equation}
\label{eq:smnmodel}
Y_i = x_i^{T}\beta + \sigma\epsilon_i \;,
\end{equation}
where $x_i$ is a $p \times 1$ vector of known covariates associated
with $Y_i$, $\beta$ is a $p \times 1$ vector of unknown regression
coefficients, $\sigma \in (0,\infty)$ is an unknown scale parameter,
and $\epsilon_1,\ldots,\epsilon_m$ are iid errors. The standard
assumption that the errors are Gaussian is often inappropriate, e.g.,
when the data contain outliers. Various heavy-tailed alternatives can
be constructed as scale mixtures of the Gaussian density. Consider an
error density of the form
\begin{equation}
\label{eq:smndensity}
f_H(\epsilon) = \int_0^\infty
\frac{\sqrt{z}}{\sqrt{2\pi}}\exp\left\lbrace -
\frac{z}{2}\epsilon^2\right\rbrace \, dH(z) \;,
\end{equation}
where $H$ is the distribution function of some non-negative random
variable.  By varying the mixing distribution $H$, many symmetric and
unimodal distributions can be constructed.  Thus, datasets with
various types of tail behavior (particularly with heavier tails than
the normal) are often modeled by choosing a distribution from this
class. In this section, we consider a Bayesian analysis of the linear
regression model ~\eqref{eq:smnmodel} when the errors
$\epsilon_1,\ldots,\epsilon_m$ are iid random variables with the
general scale mixture of normals density $f_H$ given in
~\eqref{eq:smndensity}.  There are several different prior
distributions available that lead to conditional distributions with
standard forms. \citet{hobert2018convergence} consider a standard
improper prior and show that a DA algorithm is available.  A DA
algorithm is also available in the case where we specify a proper
conditionally conjugate prior on $(\beta,\sigma^2)$ by setting $\beta
| \sigma^2 \sim \mbox{N}_p(\mu,\sigma^2 \Sigma)$ and $\sigma^2 \sim
\mbox{IG}(\alpha,\gamma)$.  Throughout this section, we will instead
consider the proper prior which takes $\beta$ and $\sigma^2$ to be
\textit{a priori} independent with $\beta \sim \mbox{N}_p
(\mu,\Sigma)$ and $\sigma^2 \sim \mbox{IG}(\alpha,\gamma)$.  This
slight change to the prior makes the DA algorithm difficult to
implement, but the HS Gibbs sampler is a viable alternative.  We now
provide the details.

Let $y = (y_1,\ldots,y_m)$ denote the observed data.  Let $X$ denote
the $m \times p$ matrix whose $i$th row is $x_i^{T}$.  We assume
throughout that $m \ge \max\{2,p\}$.  We also assume that $H$ has a
density, $h$, with respect to Lebesgue measure on $\mathbb{R}_+$.
Letting $p_H(y|\beta,\sigma^2)$ denote the joint density of the data
from the linear regression model, the posterior density is given by
\begin{align*}
\pi(\beta,\sigma^2|y) & \propto p_H(y|\beta,\sigma^2)
\pi(\beta,\sigma^2) \\ & \propto \Bigg[ \prod_{i=1}^m
  \frac{1}{\sigma}f_H\left(\frac{y_i - x_i^{T} \beta}{\sigma}\right)
  \Bigg] \pi(\beta,\sigma^2) \\ & \propto \Bigg[ \prod_{i=1}^m
  \int_{\reals_+} \frac{\sqrt{z_i}}{\sqrt{2\pi\sigma^2}}
  \exp\left\lbrace
  -\frac{z_i}{2}\frac{(y_i-x_i^{T}\beta)^2}{\sigma^2}\right\rbrace
  h(z_i) \, dz_i \Bigg] \\ & \hspace*{8mm} \times
(\sigma^2)^{-\alpha-1}\exp \left\lbrace -\frac{\gamma}{\sigma^2}
\right\rbrace \exp \left\lbrace
-\frac{(\beta-\mu)^T\Sigma^{-1}(\beta-\mu)}{2} \right\rbrace
I_{\mathbb{R}^+}(\sigma^2) \;.
\end{align*}
Define the complete data posterior density as
\begin{align*}
\pi(\beta,\sigma^2,z|y) &= \prod_{i=1}^m
\frac{\sqrt{z_i}}{\sqrt{2\pi\sigma^2}} \exp\left\lbrace
-\frac{z_i}{2}\frac{(y_i-x_i^{T}\beta)^2}{\sigma^2}\right\rbrace
h(z_i) \\ & \times (\sigma^2)^{-\alpha-1}\exp \left\lbrace
-\frac{\gamma}{\sigma^2} \right\rbrace \exp \left\lbrace
-\frac{(\beta-\mu)^T\Sigma^{-1}(\beta-\mu)}{2} \right\rbrace
I_{\mathbb{R}^+}(\sigma^2) \;,
\end{align*}
and note that $\int_{\reals_+^m} \pi(\beta,\sigma^2,z|y) \, dz =
\pi(\beta,\sigma^2|y)$, so that $z = (z_1,\ldots,z_m)$ constitutes
latent data.  We now state the conditional densities needed for the HS
Gibbs sampler.  First, conditional on $(\beta, \sigma^2, y)$, $z_1,
\ldots, z_m$ are independent, and the conditional density of $z_i$
given $(\beta, \sigma^2, y_i)$ is given by
\begin{equation}
\label{eq:zcond}
  \pi(z_i|\beta,\sigma^2,y_i) \propto z_i^{\frac{1}{2}} \exp \left\{ -
  \frac{z_i}{2} \frac{(y_i - x_i^T \beta)^2}{\sigma^2} \right\} h(z_i)
  \;.
\end{equation}
In some cases, this density turns out to be standard.  For example,
when $h$ is a gamma density, then so is $\pi(z_i|\beta,\sigma^2,y_i)$,
and when $h$ is inverted gamma, then $\pi(z_i|\beta,\sigma^2,y_i)$ is
generalized inverse Gaussian.  Even when it's not a standard density,
as long as one can make draws from $h$, then $h$ can be used as the
candidate in a simple rejection sampler.

Next, let $Q$ be an $m \times m$ diagonal matrix whose $i$th diagonal
element is $z_i^{-1}$.  We have
\[
\sigma^2 \; | \; \beta, z, y \sim \text{IG}\left( \frac{m}{2} +
\alpha, \frac{(y-X\beta)^T Q^{-1} (y-X\beta) + 2\gamma}{2} \right) \;.
\]
Finally, $\beta \; | \; \sigma^2, z, y \sim \mbox{N}_p \left( \mu',
\sigma^2 \Sigma' \right)$, where
\[
\mu' = \left( X^T Q^{-1} X + \sigma^2 \Sigma^{-1} \right)^{-1} \left(
X^T Q^{-1} y + \sigma^2 \Sigma^{-1} \mu \right) \hspace{5mm}
\mbox{and} \hspace{5mm} \Sigma' = \left( X^T Q^{-1} X + \sigma^2
\Sigma^{-1} \right)^{-1} \;.
\]
The HS Gibbs sampler is based on the Markov chain $\Phi = \{
(\beta_n,\sigma_n^2)\}_{n=0}^{\infty}$ with state space $\mathsf{X} =
\reals^p \times \reals_+$ and selection probability $r \in (0,1).$ The
dynamics of $\Phi$ are defined by the following two step procedure for
moving from $(\beta_n,\sigma_n^2) = (\beta,\sigma^2)$ to
$(\beta_{n+1},\sigma_{n+1}^2).$

\newpage
\baro \vspace*{2mm}
\noindent {\rm Iteration $n+1$ of the hybrid scan Gibbs sampler:}
\begin{enumerate}
\item Draw $Z_1, \ldots, Z_m$ independently with $$Z_i \sim \text{the
  density proportional to} \; z_i^{\frac{1}{2}} \exp \left\{ -
  \frac{z_i}{2} \frac{(y_i - x_i^T \beta)^2}{\sigma^2} \right\} h(z_i)
  \;,$$ call the observed values $z = (z_1,\ldots,z_m),$ and,
  independently, draw $W \sim \text{Uniform}(0,1).$

\item 
\begin{enumerate}
\item If $W \le r,$ draw $$\sigma^{*2} \sim \text{IG}\left(
  \frac{m}{2} + \alpha, \frac{(y-X\beta)^T Q^{-1} (y-X\beta) +
    2\gamma}{2} \right),$$ and set $(\beta_{n+1},\sigma_{n+1}^2) =
  (\beta,\sigma^{*2}).$

\item Otherwise if $r < W \le 1,$ draw $$\beta^* \sim \mbox{N}_p
  \left( \mu', \sigma^2 \Sigma' \right),$$ and set
  $(\beta_{n+1},\sigma_{n+1}^2) = (\beta^*,\sigma^2).$
\end{enumerate}
\end{enumerate}
\vspace*{-3mm}
\barba
\bigskip

We now provide convergence rate results for this HS algorithm and the
corresponding SS Gibbs sampler.  Let $\hat{\Phi} = \{
(\hat{\beta}_n,\hat{\sigma}_n^2)\}_{n=0}^{\infty}$ denote the Markov
chain defined by the following Mtd:
\[
k_{\text{G}} (\beta, \sigma^2 | \hat{\beta},\hat{\sigma}^2) =
\int_{\mathbb{R}_+^n} \pi(\beta|\sigma^2,z,y)
\pi(\sigma^2|\hat{\beta},z,y) \pi(z|\hat{\beta},\hat{\sigma}^2,y) \,
dz \;.
\]
Of course, this is just the Markov chain that one is left with when
one runs the three-block SS Gibbs sampler and ignores the latent data.
It is well known that this chain has exactly the same convergence rate
as the SS Gibbs chain.  The following result, which is proven in
Appendix~\ref{app:proof}, provides sufficient conditions for each of
these algorithms to be geometrically ergodic.

\begin{theorem}
\label{thm:smn}
The following results hold for any mixing density $h.$
\begin{enumerate}[(i)]
\item Suppose there exist constants $0 \leq \psi_1 < 1$ and $L_1 \in
  \reals$ which do not depend on $\beta$ or $\sigma^2$ such that
\begin{equation} 
\label{ine:driftgibbs}
\frac{\sum_{i=1}^m E[z_i|\beta,\sigma^2,y](y_i - x_i^T \beta)^2}{m + 2
  \alpha - 2} \leq \psi_1 \left[ \sum_{i=1}^m (y_i - x_i^T \beta)^2 +
  \beta^T \Sigma^{-1} \beta + \sigma^2 + \frac{1}{\sigma^2} \right] +
L_1 \;
\end{equation}
for every $\beta \in \reals^p, \sigma^2 \in \reals_+$.  Then
$\hat{\Phi}$ is geometrically ergodic.

\item Suppose there exist constants $\psi_2 \in \reals_+$, $0 \leq
  \psi_3 < 1$, and $L_2 \in \reals$ which do not depend on $\beta$ or
  $\sigma^2$ such that
\begin{equation} 
\label{ine:drifthybrid}
\frac{\sum_{i=1}^m E[z_i|\beta,\sigma^2,y](y_i - x_i^T \beta)^2}{m + 2
  \alpha - 2} \leq \psi_2 \Bigg[ \sum_{i=1}^m (y_i - x_i^T \beta)^2 +
  \, \beta^T\Sigma^{-1}\beta \Bigg] + \psi_3 \left( \sigma^2 +
\frac{1}{\sigma^2} \right) + L_2 \;
\end{equation}
for every $\beta \in \reals^p, \sigma^2 \in \reals_+$.  Then $\Phi$ is
geometrically ergodic for all $r \in (0,1)$.
\end{enumerate}
\end{theorem}

\begin{remark}
Note that if \eqref{ine:driftgibbs} holds, then
\eqref{ine:drifthybrid} holds with $\psi_2 = \psi_3 = \psi_1$, and
$L_2 = L_1$.  So the sufficient condition for geometric ergodicity of
the HS Gibbs algorithm is weaker than the corresponding sufficient
condition for the SS Gibbs sampler.  Of course, we are dealing with
sufficient conditions here, so by no means does Theorem~\ref{thm:smn}
imply that HS Gibbs algorithm is geometrically ergodic more often than
the SS Gibbs sampler.  On the other hand, in a given situation, if it
is known that the HS algorithm is geometrically ergodic, and it is
unknown whether or not the same is true of the SS Gibbs sampler, then
one should probably use the HS algorithm.
\end{remark}

In order to actually apply Theorem~\ref{thm:smn}, we must specify $h$
so that we can calculate (or at least bound)
$E[z_i|\beta,\sigma^2,y]$.  For example, suppose that $h$ is a
$\text{Gamma}(\frac{\nu}{2},\frac{\nu}{2})$ density, which leads to a
Student's $t$ distribution with $\nu$ degrees of freedom for the
regression errors.  In this case, $z_i|(\beta,\sigma^2,y)$ is
\[
\mbox{Gamma} \bigg( \frac{\nu+1}{2}, \frac{(y_i - x_i^T \beta)^2 + \nu
  \sigma^2}{2\sigma^2} \bigg) \;,
\]
and
\[
E[z_i|\beta,\sigma^2,y] = \frac{\sigma^2 (\nu+1)}{(y_i - x_i^T
  \beta)^2 + \nu \sigma^2} \;.
\]
It follows that \eqref{ine:drifthybrid} is satisfied since
\begin{align}
  \label{eq:ia}
\frac{\sum_{i=1}^m E[z_i|\beta,\sigma^2,y](y_i - x_i^T \beta)^2}{m + 2
  \alpha - 2} & = \frac{1}{m + 2 \alpha - 2} \sum_{i=1}^m
\frac{\sigma^2 (\nu+1)(y_i - x_i^T \beta)^2}{(y_i - x_i^T \beta)^2 +
  \nu \sigma^2} \nonumber \\ & \le \frac{\nu+1}{\nu(m + 2 \alpha - 2)}
\sum_{i=1}^m (y_i - x_i^T \beta)^2 \;.
\end{align}
Thus, Theorem~\ref{thm:smn} implies that the HS Markov chain is
geometrically ergodic (without any additional assumptions).

Unfortunately, this argument doesn't work for the SS Gibbs sampler.
Indeed, \eqref{eq:ia} doesn't establish that \eqref{ine:driftgibbs} is
satisfied unless we make the additional assumption that that $\nu >
1/(m+2\alpha-3)$.  However, another upper bound on the left-hand side
of \eqref{eq:ia} is as follows:
\begin{equation}
  \label{eq:ib}
  \frac{\sum_{i=1}^m E[z_i|\beta,\sigma^2,y](y_i - x_i^T \beta)^2}{m +
    2 \alpha - 2} \le \frac{m(\nu+1)}{m + 2 \alpha - 2} \sigma^2 \;.
\end{equation}
Now, \eqref{eq:ib} will establish \eqref{ine:driftgibbs} if $\nu <
(2\alpha-2)/m$.  So Theorem~\ref{thm:smn} implies that the SS Gibbs
chain is geometrically ergodic if either $\nu > 1/(m+2\alpha-3)$ or
$\nu < (2\alpha-2)/m$.  Of course, if $1/(m+2\alpha-3) <
(2\alpha-2)/m$, then at least one of these two inequalities must hold.
However, when $\alpha$ is small, this is not the case.

Consider a second example where $h$ is taken to be an
$\text{IG}(\alpha,1)$ density.  Under this mixing density, the
regression errors have a generalized hyperbolic distribution, which
has tails that are heavier than Gaussian, but lighter than Student's
$t$ \citep[see, e.g.,][]{jung:hobe:2014}.  In this case,
$z_i|(\beta,\sigma^2,y)$ is
\[
\mbox{GIG} \bigg( \frac{1}{2} - \alpha, \frac{(y_i - x_i^T
  \beta)^2}{\sigma^2}, 2 \bigg) \;,
\]
and
\[
E[z_i|\beta,\sigma^2,y] = \frac{\sqrt{2 \sigma^2}}{\sqrt{(y_i - x_i^T
  \beta)^2}} \frac{\mbox{K}_{- \alpha + 3/2} \Big( \sqrt{\frac{2(y_i -
      x_i^T \beta)^2}{\sigma^2}} \Big)}{\mbox{K}_{- \alpha + 1/2}
  \Big( \sqrt{\frac{2(y_i - x_i^T \beta)^2}{\sigma^2}} \Big)} \;.
\]
\citet[][p. 62]{jungdiss} shows that
\[
\frac{\mbox{K}_{- \alpha + 3/2} \Big( \sqrt{\frac{2(y_i - x_i^T
      \beta)^2}{\sigma^2}} \Big)}{\mbox{K}_{- \alpha + 1/2} \Big(
  \sqrt{\frac{2(y_i - x_i^T \beta)^2}{\sigma^2}} \Big)} \le 1 +
\frac{\sqrt{\sigma^2}}{\sqrt{2(y_i - x_i^T \beta)^2}} \;.
\]
Hence, letting $C>0$ be an arbitrary positive constant, we have
\begin{align}
  \label{eq:iga}
\frac{\sum_{i=1}^m E[z_i|\beta,\sigma^2,y](y_i - x_i^T \beta)^2}{m + 2
  \alpha - 2} & \le \frac{1}{m + 2\alpha - 2} \sum_{i=1}^m \bigg[
  \frac{\sqrt{2 \sigma^2}}{C} \sqrt{C(y_i - x_i^T \beta)^2} + \sigma^2
  \bigg] \nonumber \\ & \le \frac{1}{m + 2\alpha - 2} \sum_{i=1}^m
\bigg[\frac{C (y_i - x_i^T \beta)^2}{2} + \frac{\sigma^2}{C} +
  \sigma^2 \bigg] \nonumber \\ & = \frac{C}{2(m + 2\alpha - 2)}
\sum_{i=1}^m (y_i - x_i^T \beta)^2 + \frac{m(C+1)}{C(m + 2\alpha - 2)}
\sigma^2 \;.
\end{align}
If $\alpha>1$, then we can find $C>0$ such that $\frac{m(C+1)}{C(m +
  2\alpha - 2)} < 1$.  Therefore, Theorem~\ref{thm:smn} implies that
the HS Gibbs chain is geometrically ergodic whenever $\alpha>1$.

Now, if we can find a single value of $C>0$ such that $m(C+1) < C(m +
2\alpha - 2)$ \textit{and} $C<2(m + 2\alpha - 2)$, then \eqref{eq:iga}
will imply that \eqref{ine:driftgibbs} holds.  The existence of such a
$C$ is equivalent to $\alpha$ and $m$ satisfying the following
inequality
\[
8 \alpha^2 + \alpha(4m-16) + 8 - 5m > 0 \;.
\]
Thus, Theorem~\ref{thm:smn} implies that the SS Gibbs chain is
geometrically ergodic if $\alpha > \big( 4-m+\sqrt{m(m+2)} \big)/4$.
This inequality holds for all $\alpha \ge \frac{5}{4}$, regardless of
the value of $m$, and it does hold for smaller values of $\alpha$ when
$m$ is fixed.  For example, if $m=2$, then we only need $\alpha > (1 +
\sqrt{2})/2 \approx 1.21$.

Once $h$ is specified, HSS algorithms can be created by adding
sandwich steps to the HS Gibbs sampler.  \citet{grantdiss} develops a
HSS algorithm with two sandwich steps for the case where $h$ is a
$\text{Gamma}(\frac{\nu}{2},\frac{\nu}{2})$ density.

\section{Discussion}
\label{sec:discussion}

We have introduced generic forms of the hybrid scan Gibbs sampler and
the hybrid scan sandwich algorithm, and we have shown that, under weak
regularity conditions, the latter is theoretically better than the
former.  Moreover, we have developed and studied specific versions of
these algorithms in the context of two different realistic Bayesian
hierarchical models.  It is clear that the hybrid scan algorithms are
quite flexible, and can be used in conjunction with a variety of
practical Bayesian models.  As another example, consider a
generalization of the model in~\eqref{eq:smnmodel} in which the error
density has both heavy tails and \textit{skewness}.
\citet{da2011skew} define a skew scale mixture of normal densities,
$f_{H,\lambda}(\epsilon)$, by
\begin{equation}
f_{H,\lambda}(\epsilon) = 2 f_H(\epsilon) \Phi(\lambda \epsilon)
\;,
\end{equation}
where $f_H(\epsilon)$ is the scale mixture of normal densities defined
at~\eqref{eq:smndensity}, $\Phi(\cdot)$ is the standard normal
cumulative distribution function, and $\lambda \in \mathbb{R}$ is a
fixed parameter that controls the skewness.  Combining the associated
likelihood with the same conjugate normal/inverse-gamma prior employed
in Section~\ref{sec:smn} gives rise to a posterior distribution that
is even more unwieldy than the one studied in Section~\ref{sec:smn}.
However, \citet{jungdiss} shows that there exist two sets of latent
variables, $z=(z_1,\ldots,z_m)$ and $t=(t_1,\ldots,t_m)$,
conditionally independent of one another given $(\beta,\sigma^2,y)$,
that give rise to a complete data posterior with the following
conditionals.  Conditional on $(\beta,\sigma^2,y)$, $z_1,\ldots,z_m$
are independent, and the density of $z_i$ given $(\beta,\sigma^2,y)$
is the same as~\eqref{eq:zcond}.  Also, conditional on
$(\beta,\sigma^2,y)$, $t_1,\ldots,t_m$ are independent, and the
density of $t_i$ given $(\beta,\sigma^2,y)$ is truncated normal.
Finally, $\sigma^2$ given $(\beta, z, t, y)$ is inverse-gamma, and
$\beta$ given $(\sigma^2,z,t,y)$ is multivariate
normal. Unfortunately, the distribution of $\sigma^2$ given $(z,t,y)$
is not available in closed form, so that the DA algorithm is not
straightforward to apply.  Each iteration of the HS Gibbs algorithm
proceeds, as usual, by updating (all of) the latent data, and updating
either $\beta$ or $\sigma^2$, depending on the outcome of the flip of
an $r$-coin.

Lastly, we reiterate that, as far as theoretical convergence rates go,
it is now generally accepted that one should, if possible, use a Monte
Carlo Markov chain that converges at a geometric rate, or at least a
rate fast enough to ensure that the corresponding MCMC estimators are
asymptotically normal \citep[see, e.g.,][]{roberts1998markov}.  Hence,
even if an alternative MCMC algorithm (such as SS or RS Gibbs) appears
marginally better than a geometrically ergodic HS Gibbs sampler
according to empirical measures, that algorithm should not be favored
over the HS Gibbs algorithm \textit{unless} it is known that the
alternative has an acceptably fast convergence rate.  At present, it
appears that the convergence rates of alternative MCMC algorithms for
the family of posteriors considered in Section~\ref{sec:glmm} (and, to
some extent, those considered in Section~\ref{sec:smn}) are not known.
In situations such as these, we recommend HS Gibbs for practical use
and remind the reader that it is no more difficult to implement than
its SS or RS counterparts.

\vspace*{6mm}

\noindent {\bf \large Acknowledgment}.  The second and fourth authors
were supported by NSF Grant DMS-15-11945.

\vspace*{8mm} 
\newpage
\noindent {\LARGE \bf Appendices}
\begin{appendix}

\vspace*{-3mm}

\section{Theory for the HSS Algorithm}
\label{app:theory}

We begin with some requisite background material on Markov
operators. In keeping with the notation in the Introduction, the
target density, $f_{X,Y}(x,y),$ can be used to define an inner
product
\[
\langle g_1, g_2 \rangle_{L^2_0} = \int_{\mathsf{X}} \int_{\mathsf{Y}}
g_1(x,y) g_2(x,y) f_{X,Y}(x,y) \, dy \, dx \;,
\]
and norm $\norm{g} = \sqrt{\langle g,g \rangle}$ on the Hilbert
space 
\[
L^2_0 = \left\{ g: \mathsf{X} \times \mathsf{Y} \to \reals :
\int_{\mathsf{X}} \int_{\mathsf{Y}} g^2(x,y) f_{X,Y}(x,y) \, dy \, dx
< \infty \; \text{and} \; \int_{\mathsf{X}} \int_{\mathsf{Y}} g(x,y)
f_{X,Y}(x,y) \, dy \, dx = 0 \right\} \;.
\]
To keep things simple, we assume throughout that $f_{X,Y}(x,y)$ is a
density with respect to Lebesgue measure, but we note that the results
actually hold much more generally - see, e.g., the set-up in
\citet{khare2011spectral}. The Mtd $k_1$ corresponds to a Markov
operator $K_1: L^2_0 \to L^2_0$ that takes $g \in L^2_0$ into
\[
(K_1 g)(x,y) = \int_{\mathsf{X}} g(x',y) k_1(x'|x;y) \, dx' \;.
\]
Now, if we define $K_2$ using $k_2$ in an analogous way, then it is
clear that the Markov operator associated with the HSS algorithm,
$K_{\text{HSS}} : L^2_0 \to L^2_0,$ is given by $K_{\text{HSS}} = r
K_1 + (1-r) K_2$, where $r \in (0,1)$ is the selection
probability.  Here is our first result.

\begin{proposition}
The Markov chain underlying the HSS algorithm is reversible.
\end{proposition}

\begin{proof}
It suffices to show that $K_{\text{HSS}}$ is a self-adjoint
operator. We start by showing that $K_1$ is self-adjoint. First, it's
easy to see that $f_{X,Y}(x,y) k_1(x'|x;y) = f_{X,Y}(x',y)
k_1(x|x';y).$ It follows that
\begin{align*}
f_{X,Y}(x,y)(K_1 g)(x,y) &= f_{X,Y}(x,y) \int_{\mathsf{X}} g(x',y)
k_1(x'|x;y) \, dx' \\ &= \int_{\mathsf{X}} g(x',y) f_{X,Y}(x,y)
k_1(x'|x;y) \, dx' \\ &= \int_{\mathsf{X}} g(x',y) f_{X,Y}(x',y)
k_1(x|x';y) \, dx'
\end{align*}
Thus,
\begin{align*}
\langle K_1 g, h \rangle_{L^2_0} &= \int_{\mathsf{X}}
\int_{\mathsf{Y}} (K_1 g)(x,y) h(x,y) f_{X,Y}(x,y) \, dy \, dx \\ &=
\int_{\mathsf{X}} \int_{\mathsf{Y}} h(x,y) \left[ \int_{\mathsf{X}}
  g(x',y) f_{X,Y}(x',y) k_1(x|x';y) \, dx' \right] \, dy \, dx \\ &=
\int_{\mathsf{X}} \int_{\mathsf{Y}} \left[ \int_{\mathsf{X}} h(x,y)
  k_1(x|x';y) \, dx \right] g(x',y) f_{X,Y}(x',y) \, dy \, dx' \\ &=
\langle g, K_1 h \rangle_{L^2_0} \;,
\end{align*}
where the third equality follows from Fubini's theorem. Now an
analogous argument shows that $K_2$ is self-adjoint, and it follows
immediately that $r K_1 + (1-r) K_2$ is also self-adjoint.
\end{proof}

We now look more closely at the two criteria for comparing MCMC
algorithms that were mentioned in the Introduction: rate of
convergence and asymptotic variance.  Let $\Phi =
\{(X_n,Y_n)\}_{n=0}^{\infty}$ denote a generic Markov chain on
$\mathsf{X} \times \mathsf{Y}$ that is reversible with respect to
$f_{X,Y}$.  Assume further that $\Phi$ is Harris ergodic; that is,
aperiodic, irreducible and Harris recurrent.  Let $K$ denote the
corresponding Markov operator on $L^2_0$.  Let $L^2_{0,1} \subset
L^2_0$ denote the functions for which
\[
\int_{\mathsf{X}} \int_{\mathsf{Y}} g^2(x,y) f_{X,Y}(x,y) \, dy \, dx
= 1 \;.
\]
The norm of the operator $K$ is defined as
\[
\norm{K} = \sup_{g \in  L^2_{0,1}} \norm{K g} \;.
\]
(Since $K$ is self-adjoint, we also have $\norm{K} = \sup_{g \in
  L^2_{0,1}} |\langle K g , g \rangle_{L^2_0}|$.)  The quantity
$\norm{K},$ which takes values in $[0,1]$, represents the convergence
rate of $\Phi,$ with smaller values associated with faster
convergence.  In fact, $\Phi$ is geometrically ergodic if and only if
$\norm{K} < 1$ \citep{roberts1997geometric}. One way to choose between
two MCMC algorithms for the same problem is to favor the one whose
Markov operator has smaller norm.

Now let $g : \mathsf{X} \times \mathsf{Y} \to \reals$ be (non-constant
and) such that
\[
\int_{\mathsf{X}} \int_{\mathsf{Y}} g^2(x,y) f_{X,Y}(x,y) \, dy \, dx
< \infty \;.
\]
Let $\theta = \int_{\mathsf{X}} \int_{\mathsf{Y}} g(x,y) f_{X,Y}(x,y)
\, dy \, dx$, and let $\overline{g}_n = \frac{1}{n} \sum_{i=0}^{n-1}
g(X_n,Y_n)$.  If $\Phi$ is geometrically ergodic, then the Markov
chain CLT implies that there exists $\sigma^2_{g,K} \in (0,\infty)$
such that, as $n \to \infty,$ $\sqrt{n} (\overline{g}_n - \theta)
\cvgindist \text{N}(0,\sigma^2_{g,K}).$ If $g$ is square integrable
with respect to $f_{X,Y},$ but the CLT does not hold, then set
$\sigma^2_{g,K} = \infty.$ Suppose $\Phi^*$ is a second Markov chain
(with corresponding operator $K^*$) that satisfies all the properties
we have assumed $\Phi$ satisfies. If $\sigma^2_{g,K^*} <
\sigma^2_{g,K}$ for all square integrable $g,$ then we say that $K^*$
is more efficient than $K,$ and we write $K^* \succeq_E K$.

Before we can state the main result, we must define a few more
operators. First, let $\hat{L}^2_0$ denote the space of functions that
are square integrable and have mean zero with respect to
$f_{Y,Z}(y,z).$ We denote the inner product on this space by $\langle
\cdot, \cdot \rangle_{\hat{L}^2_0}.$ The Mtf $R_1(z,dz';y)$ defines an
operator $R_1 : \hat{L}^2_0 \to \hat{L}^2_0$ that takes $h \in
\hat{L}^2_0$ to $$(R_1 h)(y,z) = \int_{\mathsf{Z}} h(y,z')
R_1(z,dz';y) \;.$$ It follows immediately from ~\eqref{eq:reverse}
that $R_1$ is self-adjoint (with respect to $f_{Y,Z}$). Of course,
$R_1$ is a positive operator if $\langle R_1 h, h
\rangle_{\hat{L}^2_0} \geq 0$ for all $h \in \hat{L}^2_0.$ Let $R_2$
denote the analogous operator corresponding to the Mtf $R_2$, and let
$K_{\text{HS}}$ denote the Markov operator (on $L^2_0$) corresponding
to the HS Gibbs sampler.

\begin{proposition}
  \label{prop:main_comp}
  Suppose the Markov chains associated with $K_{\text{HSS}}$ and
  $K_{\text{HS}}$ are both Harris ergodic. Then $K_{\text{HSS}}
  \succeq K_{\text{HS}}$. If, in addition, $R_1$ and $R_2$ are both
  positive operators, then $\norm{K_{\text{HSS}}} \leq
  \norm{K_{\text{HS}}}.$
\end{proposition}

\begin{proof}
Fix $g \in L^2_0$ and define $$g^*(y,z) = \int_{\mathsf{X}} g(x,y)
f_{X|Y,Z}(x|y,z) \, dx \;.$$ It's easy to see that $g^* \in
\hat{L}^2_0.$ Now
\begin{align}
&\langle K_1 g , g \rangle_{L^2_0} \nonumber \\ &= \int_{\mathsf{X}}
  \int_{\mathsf{Y}} (K_1 g)(x,y) g(x,y) f_{X,Y}(x,y) \, dy \, dx
  \nonumber \\ &= \int_{\mathsf{X}} \int_{\mathsf{Y}} \left[
    \int_{\mathsf{X}} g(x',y) \int_{\mathsf{Z}} \int_{\mathsf{Z}}
    f_{X|Y,Z}(x'|y,z') R_1(z,dz';y) f_{Z|X,Y}(z|x,y) \, dz \, dx'
    \right] g(x,y) f_{X,Y}(x,y) \, dy \, dx \nonumber \\ &=
  \int_{\mathsf{X}} \int_{\mathsf{Y}} \int_{\mathsf{X}}
  \int_{\mathsf{Z}} \int_{\mathsf{Z}} g(x',y) f_{X|Y,Z}(x'|y,z')
  R_1(z,dz';y) f_{Z|X,Y}(z|x,y) g(x,y) f_{X,Y}(x,y) \, dz \, dx' \, dy
  \, dx \nonumber \\ &= \int_{\mathsf{Y}} \int_{\mathsf{Z}} \left[
    \int_{\mathsf{Z}} g^*(y,z') R_1(z,dz';y) \right] g^*(y,z)
  f_{Y,Z}(y,z) \, dy \, dz \nonumber \\ &= \langle R_1 g^*, g^*
  \rangle_{\hat{L}^2_0} \;.
\label{eq:cov}
\end{align}
Note that $\langle R_1 g^*, g^* \rangle_{\hat{L}^2_0}$ is the
covariance of $g^*(Y_0,Z_0)$ and $g^*(Y_1,Z_1)$ where
${(Y_n,Z_n)}_{n=0}^{\infty}$ is the \textit{stationary} version of the
Markov chain driven by $R_1$ (so $(Y_0,Z_0) \sim f_{Y,Z}$). Let
$\tilde{K}_1$ denote $K_1$ when $R_1$ is trivial. Then $\langle
\tilde{K}_1 g , g \rangle_{L^2_0} = \langle g^*, g^*
\rangle_{\hat{L}^2_0},$ which is the variance of $g^*(Y_0,Z_0)$ when
$(Y_0,Z_0) \sim f_{Y,Z}.$ Hence by Cauchy-Schwarz, $$\langle K_1 g , g
\rangle_{L^2_0} = \langle R_1 g^*, g^* \rangle_{\hat{L}^2_0} \leq
\langle g^*, g^* \rangle_{\hat{L}^2_0} = \langle \tilde{K}_1 g , g
\rangle_{L^2_0} \;.$$ An analogous argument shows that $\langle K_2 g
, g \rangle_{L^2_0} \leq \langle \tilde{K}_2 g , g \rangle_{L^2_0},$
where $\tilde{K}_2$ denotes $K_2$ with a trivial $R_2$.  Of course,
$K_{\text{HS}} = r \tilde{K}_1 + (1-r) \tilde{K}_2$.  Therefore, for
any $g \in L^2_0,$ we have
\begin{equation}
\label{dsinequality}
\langle K_{\text{HSS}} g , g \rangle_{L^2_0} = \langle (r K_1 + (1-r)
K_2) g , g \rangle_{L^2_0} \leq \langle (r \tilde{K}_1 + (1-r)
\tilde{K}_2) g , g \rangle_{L^2_0} = \langle K_{\text{HS}} g , g
\rangle_{L^2_0} \;,
\end{equation}
and it now follows from results in \citet{mira1999ordering} that
$K_{\text{HSS}} \succeq_E K_{\text{HS}}$.

Now, if $R_1$ is positive, then it follows immediately from
~\eqref{eq:cov} that $K_1$ is also positive. Of course, in an
analogous manner, positivity of $R_2$ implies that of $K_2.$ Then
since $K_{\text{HSS}}$ and $K_{\text{HS}}$ are both self-adjoint, it
follows from \eqref{dsinequality} that $\norm{K_{\text{HSS}}} \leq
\norm{K_{\text{HS}}}$.
\end{proof}

\begin{remark}
As explained in \citet{mira1999ordering}, generally fast convergence
and small asymptotic variance are conflicting goals. Indeed, a Markov
chain has a small norm when the spectrum of its operator is
concentrated near zero, whereas small asymptotic variance is
associated with a spectrum that is concentrated near -1.  When $R_1$
and $R_2$ are both positive operators, then $K_{\text{HSS}}$ and
$K_{\text{HS}}$ are also positive, which implies that their spectra
are both subsets of $[0,1]$.  In this context, fast convergence and
small asymptotic variance are both associated with a spectrum
concentrated near zero, and are no longer conflicting goals.
\end{remark}

\section{Proof of Theorem~\ref{thm:smn}}
\label{app:proof}

We begin with several lemmas. The following lemma is proved in
\citet{khare2011spectral}.

\begin{lemma}
\label{lemma:kh}
Fix $m \in \{2,3,\ldots\}$ and $p \in \mathbb{N},$ and let
$t_1,\ldots,t_m$ be vectors in $\reals^p.$ Then
\[
C_{p,m}(t_1 ; t_2, \ldots, t_m) := \sup_{c \in \reals^m_+} t_1^T
\left( t_1 t_1^T + \sum_{i=2}^m c_i t_i t_i^T + c_1 I \right)^{-2}
t_1
\]
is finite.
\end{lemma}

For a symmetric matrix $M$, let $\lambda^*\{M\}$ denote the largest
eigenvalue of $M$, and define the matrix norm as follows
\[
\norm{M} = \sup_{\norm{x}=1} \norm{Mx} = \sup_{\norm{x}=1} \sqrt{x^T
  M^2 x} \;.
\]
The following result is easily established.

\begin{lemma}
\label{lemma:matrix}
If $A$ is a symmetric, non-negative definite matrix, then
\[
\norm{(I + A)^{-1}} = \lambda^* \big\{ (I + A)^{-1} \big\} \le 1 \;,
\]
and $I - (I + A)^{-1}$ is non-negative definite.
\end{lemma}

Let $\{y_i\}_{i=1}^m$ and $\{x_i\}_{i=1}^m$ be the data and the
covariates, respectively, from the model in Section~\ref{sec:smn}.

\begin{lemma}
\label{lemma:drift}
Define $V: \mathbb{R}^p \times \mathbb{R}_+ \rightarrow (0,\infty)$ as
follows
\[
V(\beta, \sigma^2) = \sum_{i=1}^m (y_i - x_i^T \beta)^2 + \beta^T
\Sigma^{-1} \beta + \sigma^2 + \frac{1}{\sigma^2} \;.
\]
The function $V$ is unbounded off compact sets, i.e., the sublevel
sets of $V$ are compact.
\end{lemma}

\begin{proof}
We must show that for every $d \ge 0$, the set
\[
S_d = \bigg \{ (\beta,\sigma^2) \in \reals^p \times \reals_+ :
V(\beta, \sigma^2) = \sum_{i=1}^m (y_i - x_i^T \beta)^2 + \beta^T
\Sigma^{-1} \beta + \sigma^2 + \frac{1}{\sigma^2} \le d \bigg\}
\]
is compact.  Since $V$ is continuous, it suffices to show that
$|\beta_i|$ is bounded for all $i \in \{1,2,\ldots,p\}$ and that
$\sigma^2$ is bounded away from $0$ and $\infty.$ Since $\Sigma$ is
positive definite, $\beta^T \Sigma^{-1} \beta \le d$ implies that
$|\beta_i|$ is bounded for all $i \in \{1,2,\ldots,p\}.$ Also,
$\sigma^2 + \frac{1}{\sigma^2} \le d$ implies that $\sigma^2$ is
bounded away from $0$ and $\infty$.
\end{proof}

\begin{lemma}
\label{lemma:hybridge}
If the hybrid scan Gibbs sampler is geometrically ergodic for some
selection probability $r^* \in (0,1),$ then it is geometrically
ergodic for every selection probability $r \in (0,1)$.
\end{lemma}

\begin{proof}
The Mtf of the HS chain (with selection probability $r$) is given by
\begin{align*}
K_{\text{HS, $r$}}((x,y),A) &= r \int_{\mathsf{X}} I_A(x',y)
\int_{\mathsf{Z}} f_{X|Y,Z}(x'|y,z) f_{Z|X,Y}(z|x,y) \, dz \, dx'
\\ &+ (1-r) \int_{\mathsf{Y}} I_A(x,y') \int_{\mathsf{Z}}
f_{Y|X,Z}(y'|x,z) f_{Z|X,Y}(z|x,y) \, dz \, dy'.
\end{align*}
It is easy to show that
\[
K_{\text{HS, $r$}}((x,y),A) \geq \text{min} \left( \frac{r}{r^*},
\frac{1-r}{1-r^*} \right) K_{\text{HS, $r^*$}}((x,y),A) \;,
\]
and thus
\[
K_{\text{HS, $r$}}((x,y),A) \geq \delta' K_{\text{HS, $r^*$}}((x,y),A)
\]
for all measurable sets $A$ and all $(x,y) \in \mathsf{X} \times
\mathsf{Y}$, where $\delta' = \text{min} \left( \frac{r}{r^*},
\frac{1-r}{1-r^*} \right) > 0$.  Since the HS chain is reversible,
Theorem 1 in \citet{jones2014convergence} implies the result.
\end{proof}

\begin{proof}[Proof of Theorem~\ref{thm:smn}]
In view of Lemma~\ref{lemma:drift} above and Lemma 15.2.8 of
\citet{meyn2012markov}, in each case it suffices to verify the
geometric drift condition for the function
\[
V(\beta, \sigma^2) = \sum_{i=1}^m (y_i - x_i^T \beta)^2 + \beta^T
\Sigma^{-1} \beta + \sigma^2 + \frac{1}{\sigma^2} \;,
\]
i.e., we must show that
\[
E(V(\beta, \sigma^2) | \hat{\beta}, \hat{\sigma}^2) \le \lambda
V(\hat{\beta}, \hat{\sigma}^2) + L
\]
for some constants $\lambda \in [0,1)$ and $L \in \reals$, where for
  part (i) of the theorem the expectation is taken with respect to the
  Mtf of the SS Gibbs chain, and for part (ii) of the theorem the
  expectation is taken with respect to the Mtf of the HS chain.  We
  begin with the SS Gibbs algorithm.
\begin{align*}
E(V(\beta, \sigma^2) | \hat{\beta}, \hat{\sigma}^2) & =
\int_{\reals_+} \int_{\reals^p} V(\beta, \sigma^2) \bigg[
  \int_{\reals^m_+} \pi(\beta | \sigma^2, z, y) \pi(\sigma^2 |
  \hat{\beta}, z, y) \pi(z| \hat{\beta}, \hat{\sigma}^2, y) \, dz
  \bigg] \, d\beta \, d\sigma^2 \\ &= \int_{\reals^m_+} \left[
  \int_{\reals_+} \left\{ \int_{\reals^p} V(\beta, \sigma^2) \pi(\beta
  | \sigma^2, z, y) \, d\beta \right\} \pi(\sigma^2 | \hat{\beta}, z,
  y) \, d\sigma^2 \right] \pi(z| \hat{\beta}, \hat{\sigma}^2, y) \,
dz.
\end{align*}
We have
\begin{align*}
\sum_{i=1}^m (y_i - x_i^T \beta)^2 + \beta^T \Sigma^{-1} \beta &=
\norm{y-X\beta}^2 + \beta^T \Sigma^{-1} \beta \\ &\leq 2\norm{y}^2 +
2\norm{X\beta}^2 + \beta^T \Sigma^{-1} \beta \\ &= 2\norm{y}^2 +
2\norm{X \Sigma^{\frac{1}{2}} \Sigma^{-\frac{1}{2}} \beta}^2 + \beta^T
\Sigma^{-1} \beta \\ &\leq 2\norm{y}^2 + (2 \norm{X
  \Sigma^{\frac{1}{2}}}^2 + 1) \norm{\Sigma^{-\frac{1}{2}} \beta}^2.
\end{align*}
Let $\tilde{X} = X\Sigma^{\frac{1}{2}},$ let $\tilde{x}_i$ be the
$i$th column of $\tilde{X}^T,$ and let $\tilde{Q}$ be an $m \times m$
diagonal matrix whose $i$th diagonal element is $\sigma^2 z_i^{-1}$.
Then, given $(\sigma^2,z,y),$ $\Sigma^{-\frac{1}{2}} \beta$ is a
multivariate normal random vector with mean $(\tilde{X}^T
\tilde{Q}^{-1} \tilde{X} + I)^{-1} (\tilde{X}^T \tilde{Q}^{-1} y +
\Sigma^{-\frac{1}{2}} \mu)$, and covariance matrix $(\tilde{X}^T
\tilde{Q}^{-1} \tilde{X} + I)^{-1}.$ It follows from
Lemma~\ref{lemma:kh} that for each $i \in \{1,2,\ldots,m\}$ and for
all $z \in \reals^m_+$,
\[
\tilde{x}_i^T \left( \tilde{x}_i \tilde{x}_i^T + \sum_{j \neq i}
\frac{z_j}{z_i} \tilde{x}_j \tilde{x}_j^T + \frac{\sigma^2}{z_i} I
\right)^{-2} \tilde{x}_i \leq C_i(\tilde{X}) \;,
\]
where $C_i(\tilde{X})$ is a finite constant. Recall that if $A$ and
$B$ are symmetric matrices of the same dimension such that $A - B$ is
non-negative definite, then $\text{tr}(A) \geq \text{tr}(B).$ Then, we
have
\begin{align*}
&E\left[ \norm{\Sigma^{-\frac{1}{2}} \beta}^2 \Bigm| \sigma^2, z, y
    \right] \\ &= \norm{(\tilde{X}^T \tilde{Q}^{-1} \tilde{X} +
    I)^{-1} (\tilde{X}^T \tilde{Q}^{-1} y + \Sigma^{-\frac{1}{2}}
    \mu)}^2 + \text{tr}((\tilde{X}^T \tilde{Q}^{-1} \tilde{X} +
  I)^{-1}) \\ &\leq 2 \norm{(\tilde{X}^T \tilde{Q}^{-1} \tilde{X} +
    I)^{-1} \tilde{X}^T \tilde{Q}^{-1} y}^2 + 2 \norm{(\tilde{X}^T
    \tilde{Q}^{-1} \tilde{X} + I)^{-1} \Sigma^{-\frac{1}{2}} \mu}^2 +
  \text{tr}((\tilde{X}^T \tilde{Q}^{-1} \tilde{X} + I)^{-1}) \\ &\leq
  2 \norm{(\tilde{X}^T \tilde{Q}^{-1} \tilde{X} + I)^{-1} \tilde{X}^T
    \tilde{Q}^{-1} y}^2 + 2\norm{(\tilde{X}^T \tilde{Q}^{-1} \tilde{X}
    + I)^{-1}}^2 \norm{\Sigma^{-\frac{1}{2}} \mu}^2 +
  \text{tr}((\tilde{X}^T \tilde{Q}^{-1} \tilde{X} + I)^{-1}) \\ &\leq
  2 \norm{(\tilde{X}^T \tilde{Q}^{-1} \tilde{X} + I)^{-1} \tilde{X}^T
    \tilde{Q}^{-1} y}^2 + 2\norm{\Sigma^{-\frac{1}{2}} \mu}^2 +
  \text{tr}(I) \\ &= 2 \left\Vert\sum_{i=1}^m \left( \sum_{j=1}^m
  \frac{z_j \tilde{x}_j \tilde{x}_j^T}{\sigma^2} + I \right)^{-1}
  \frac{z_i \tilde{x}_i y_i}{\sigma^2}\right\Vert^2 +
  2\norm{\Sigma^{-\frac{1}{2}} \mu}^2 + p \\ & \leq 2 \left(
  \sum_{i=1}^m \left\Vert\left( \frac{z_i \tilde{x}_i
    \tilde{x}_i^T}{\sigma^2} + \sum_{j \neq i} \frac{z_j \tilde{x}_j
    \tilde{x}_j^T}{\sigma^2} + I \right)^{-1} \frac{z_i \tilde{x}_i
    y_i}{\sigma^2}\right\Vert \right)^2 + 2\norm{\Sigma^{-\frac{1}{2}}
    \mu}^2 + p \\ &= 2 \left( \sum_{i=1}^m |y_i| \left\Vert \left(
  \tilde{x}_i \tilde{x}_i^T + \sum_{j \neq i} \frac{z_j}{z_i}
  \tilde{x}_j \tilde{x}_j^T + \frac{\sigma^2}{z_i} I \right)^{-1}
  \tilde{x}_i \right\Vert \right)^2 + 2\norm{\Sigma^{-\frac{1}{2}}
    \mu}^2 + p \\ &= 2 \left( \sum_{i=1}^m |y_i| \sqrt{\tilde{x}_i^T
    \left( \tilde{x}_i \tilde{x}_i^T + \sum_{j \neq i} \frac{z_j}{z_i}
    \tilde{x}_j \tilde{x}_j^T + \frac{\sigma^2}{z_i} I \right)^{-2}
    \tilde{x}_i} \right)^2 + 2\norm{\Sigma^{-\frac{1}{2}} \mu}^2 + p
  \\ &\leq 2 \left( \sum_{i=1}^m |y_i| \sqrt{C_i(\tilde{X})} \right)^2
  + 2\norm{\Sigma^{-\frac{1}{2}} \mu}^2 + p,
\end{align*}
where the third inequality follows from Lemma~\ref{lemma:matrix}.
Therefore, $E\left[ \norm{\Sigma^{-\frac{1}{2}} \beta}^2 \Bigm|
  \sigma^2, z, y \right]$ is bounded above by a finite constant that
we will call $D$.  Therefore, we have
\begin{equation}
  \label{eq:t1p1}
  \int_{\reals^p} \bigg[ \sum_{i=1}^m (y_i - x_i^T \beta)^2 + \beta^T
    \Sigma^{-1} \beta \bigg] \pi(\beta | \sigma^2, z, y) \, d\beta \le
  2\norm{y}^2 + D \big( 2 \norm{X \Sigma^{\frac{1}{2}}}^2 + 1 \big)
  \;.
\end{equation}
Now, recall that $\sigma^2 | \beta, z, y \sim \mbox{IG} \Big(
\frac{m}{2} + \alpha, \frac{(y-X\beta)^T Q^{-1} (y-X\beta) +
  2\gamma}{2} \Big)$.  It follows that
\begin{equation}
  \label{eq:t1p2}
  E\big[ (\sigma^2)^{-1} | \hat{\beta},z,y \big] =
  \frac{m+2\alpha}{(y-X\hat{\beta})^T Q^{-1} (y-X\hat{\beta}) +
    2\gamma} = \frac{m+2\alpha}{\sum_{i=1}^n z_i (y_i - x_i^T
    \hat{\beta})^2 + 2\gamma} \leq \frac{m+2\alpha}{2\gamma} \;.
\end{equation}
And since $\frac{m}{2} + \alpha > 1$, we have
\begin{equation*}
  E(\sigma^2 | \hat{\beta},z,y) = \frac{(y-X\hat{\beta})^T Q^{-1}
    (y-X\hat{\beta}) + 2\gamma}{m+2\alpha-2} = \frac{\sum_{i=1}^m z_i
    (y_i - x_i^T \hat{\beta})^2 + 2\gamma}{m+2\alpha-2} \;.
\end{equation*}
Our assumption then implies that 
\begin{equation}
  \label{eq:t1p3}
\int_{\reals^m_+} \left[ \int_{\reals_+} \sigma^2 \pi(\sigma^2 |
  \hat{\beta}, z, y) \, d\sigma^2 \right] \pi(z| \hat{\beta},
\hat{\sigma}^2, y) \, dz \le \frac{2\gamma}{m+2\alpha-2} + \psi_1
V(\hat{\beta},\hat{\sigma}^2) + L_1 \;,
\end{equation}
where $\psi_1 \in [0,1)$ and $L_1 \in \mathbb{R}$.
  Combining~\eqref{eq:t1p1}, ~\eqref{eq:t1p2}, and ~\eqref{eq:t1p3} we
  have
\begin{equation*}
  E(V(\beta, \sigma^2) | \hat{\beta}, \hat{\sigma}^2) \le \psi_1
  V(\hat{\beta},\hat{\sigma}^2) + C \;,
\end{equation*}
where
\[
C = 2\norm{y}^2 + D \big( 2 \norm{X \Sigma^{\frac{1}{2}}}^2 + 1 \big)
+ \frac{m+2\alpha}{2\gamma} + \frac{2\gamma}{m+2\alpha-2} + L_1 \;,
\]
and hence the SS Gibbs Markov chain is geometrically ergodic.

Now for the HS algorithm, we have
\begin{align*}
E(V(\beta, \sigma^2) | \hat{\beta}, \hat{\sigma}^2) = r
\int_{\reals^m_+} \int_{\reals_+} & V(\hat{\beta}, \sigma^2)
\pi(\sigma^2 | \hat{\beta}, z, y) \pi(z| \hat{\beta}, \hat{\sigma}^2,
y) \, d\sigma^2 \, dz \\ & + (1-r) \int_{\reals^m_+} \int_{\reals^p}
V(\beta, \hat{\sigma}^2) \pi(\beta | \hat{\sigma}^2, z, y) \pi(z|
\hat{\beta}, \hat{\sigma}^2, y) \, d\beta \, dz \;.
\end{align*}
Equation~\eqref{eq:t1p1} implies that
\begin{align*}
\int_{\reals^m_+} \int_{\reals^p} V(\beta, \hat{\sigma}^2) \pi(\beta |
\hat{\sigma}^2, z, y) & \pi(z| \hat{\beta}, \hat{\sigma}^2, y) \,
d\beta \, dz \\ & \le 2\norm{y}^2 + D \big( 2 \norm{X
  \Sigma^{\frac{1}{2}}}^2 + 1 \big) + \hat{\sigma}^2 +
\frac{1}{\hat{\sigma}^2} \;.
\end{align*}
Equations~\eqref{eq:t1p2} and ~\eqref{eq:t1p3} imply that
\begin{align*}
\int_{\reals^m_+} \int_{\reals_+} & V(\hat{\beta}, \sigma^2)
\pi(\sigma^2 | \hat{\beta}, z, y) \pi(z| \hat{\beta}, \hat{\sigma}^2,
y) \, d\sigma^2 \, dz \le \sum_{i=1}^m (y_i - x_i^T \hat{\beta})^2 +
\hat{\beta}^T \Sigma^{-1} \hat{\beta} \; + \\ &
\frac{m+2\alpha}{2\gamma} + \frac{2\gamma}{m+2\alpha-2} +
\frac{1}{m+2\alpha-2} \sum_{i=1}^m (y_i - x_i^T \hat{\beta})^2
\int_{\reals_+} z_i \, \pi(z_i| \hat{\beta}, \hat{\sigma}^2, y) \,
dz_i \;.
\end{align*}
By assumption, we have $\psi_2 \in \mathbb{R}_+$, $\psi_3 \in [0,1)$
  and $L_2 \in \mathbb{R}$ such that
\begin{align*}
\int_{\reals^m_+} \int_{\reals_+} & V(\hat{\beta}, \sigma^2)
\pi(\sigma^2 | \hat{\beta}, z, y) \pi(z| \hat{\beta}, \hat{\sigma}^2,
y) \, d\sigma^2 \, dz \le \sum_{i=1}^m (y_i - x_i^T \hat{\beta})^2 +
\hat{\beta}^T \Sigma^{-1} \hat{\beta} \; + \\ &
\frac{m+2\alpha}{2\gamma} + \frac{2\gamma}{m+2\alpha-2} + \psi_2
\sum_{i=1}^m (y_i - x_i^T \hat{\beta})^2 + \psi_2 \hat{\beta}^T
\Sigma^{-1} \hat{\beta} + \psi_3 \Big( \hat{\sigma}^2 +
\frac{1}{\hat{\sigma}^2} \Big) + L_2 \;.
\end{align*}
Putting all of this together, we have
\begin{align*}
E(V(\beta, \sigma^2) | \hat{\beta}, \hat{\sigma}^2) \le r (\psi_2 + 1)
\bigg( \sum_{i=1}^m (y_i - x_i^T \hat{\beta})^2 \bigg) + & r (\psi_2
+1) \hat{\beta}^T \Sigma^{-1} \hat{\beta} \; + \\ & \Big[ r \psi_3 +
  (1-r) \Big] \Big( \hat{\sigma}^2 + \frac{1}{\hat{\sigma}^2} \Big) +
C' \;,
\end{align*}
where 
\begin{equation*}
C' = (1-r) \Big[ 2\norm{y}^2 + D \big( 2 \norm{X
    \Sigma^{\frac{1}{2}}}^2 + 1 \big) \Big] +
\frac{r(m+2\alpha)}{2\gamma} + \frac{2r\gamma}{m+2\alpha-2} + r L_2
\;.
\end{equation*}
Hence,
\begin{equation*}
E(V(\beta, \sigma^2) | \hat{\beta}, \hat{\sigma}^2) \le \max \Big \{
r(1+\psi_2), r \psi_3 + (1-r) \Big \} V(\hat{\beta}, \hat{\sigma}^2) +
C' \;,
\end{equation*}
and, since $r \psi_3 + (1-r) < 1$ for all $r \in (0,1)$, we have a
valid geometric drift condition as long as $r < (1+\psi_2)^{-1}$.
Finally, an appeal to Lemma~\ref{lemma:hybridge} completes the proof.
\end{proof}

\end{appendix}

\bibliographystyle{ims}
\bibliography{hybrid}

\end{document}